\documentclass[11pt, a4paper]{article}
\usepackage[letterpaper, left=2cm, right=2cm, top=2cm, bottom=2cm]{geometry}
\usepackage{graphicx}
\usepackage{multirow}
\usepackage{dsfont}
\usepackage{color}
\usepackage{appendix}
\usepackage{url}
\usepackage{wrapfig}
\usepackage{fancyhdr}
\usepackage{multirow}
\usepackage{empheq}
\usepackage{enumerate}
\usepackage{cancel}
\usepackage{ulem}
\usepackage{amssymb}
\usepackage{amsmath}
\usepackage{amsthm}

\newtheorem{thm}{Theorem}[section]

\newtheorem{lem}[thm]{Lemma}
\newtheorem{prop}[thm]{Proposition}
\newtheorem{defn}[thm]{Definition}

\newtheorem{rmk}[thm]{Remark}
\newcommand\mycom[2]{\genfrac{}{}{0pt}{}{#1}{#2}}

\title{Existence and Uniqueness of Solution of a Continuous Flow Bioreactor Model with Two Species.}
\author{ M. CRESPO$^1$, B. IVORRA$^2$ and A.M RAMOS$^3$ \\[0.3cm]
Departamento de Matem\'atica Aplicada, Universidad Complutense de Madrid \\[0.1cm]
$\&$ Instituto de Matem\'atica Interisciplinar
\\[0.3cm]
Plaza de Ciencias, 3, 28040 Madrid, Spain. \\[0.3cm]
{\small  $^1$ E-mail: mcresp01@ucm.es. Tel.: +34-913944462} \\[0.1cm]
{\small $^2$ E-mail: ivorra@mat.ucm.es. Tel.: +34-913944415} \\[0.1cm]
{\small $^3$ E-mail: angel@mat.ucm.es. Tel.: +34- 913944642 }
}
\date{}
\begin{document}
\maketitle
\begin{abstract}
In this work, we perform the mathematical analysis of a coupled system of two reaction-diffusion-advection equations and Danckwerts boundary conditions, which models the interaction between a microbial population (e.g., bacterias) and a diluted substrate (e.g., nitrate) in a continuous flow bioreactor. This type of bioreactor can be used, for instance, for water treatment. First, we prove the existence and uniqueness of solution, under the hypothesis of linear reaction by using classical results for linear parabolic boundary value problems. Next, we prove the existence and uniqueness of solution for some nonlinear reactions by applying \textit{Schauder Fixed Point Theorem} and the theorem obtained for the linear case. Results about the nonnegativeness and boundedness of the solution are also proved here. 
\end{abstract}

\section{Introduction}\label{sec:introduction}
Water treatment is an important environmental issue whose main objective is to provide clean water to human populations (see, e.g., \cite{american2003water}). One of the principal causes of contamination of water resources is due to organic or mineral substrates (e.g., nitrates or phosphorus) which are produced by the agriculture and chemical sectors. A way to perform the decontamination of these substrates is to use a bioreactor. In our framework, a bioreactor is a vessel in which a microorganism (e.g., bacteria or yeast), called biomass, is used to degrade a considered diluted substrate. Developing mathematical models that allow to simulate the interaction between  biomass and substrate inside a bioreactor is of great interest in order to design efficient water treatment devices (see, e.g., \cite{bellorivas,AIC}).

There exists many mathematical models describing the competition between biomass and substrate in bioreactors. Most theoretical studies consider a well-mixed environment, such as the chemostat (see, e.g.,~\cite{SmithWaltman}). Focusing on bacterias, some of the first explorations of bacterial growth in spatially distributed environments were carried out by Lauffenburger, Aris and Keller~\cite{Keller} and Lauffenburger and Calcagno~\cite{Calcagno}. Particularly, Kung and Baltzis~\cite{Kung} considered a tubular bioreactor (assumed to be a thin tube), through which a liquid charged with a substrate at constant concentration enters the bioreactor with a constant flow rate, and the outflow leaves the bioreactor with the same flow rate. These considerations lead to a coupled system of two reaction-diffusion-advection equations with Danckwerts boundary conditions, typically used for continuous flows bioreactors (see, e.g., \cite{Bailey,Kung,wen1975models}). 

This system of parabolic equations has received considerable attention in the literature, both from theoretical and applied points of view. One can find the one-dimensional version of the model with Danckwerts boundary conditions in \cite{Ballyk}, \cite{Drame2004} and\cite{Qiu2004465}, where the asymptotic behavior of the solution is studied under the assumption of constant fluid flow and entering substrate. There exist many works on	the existence and uniqueness of solution of linear parabolic equations~\cite{evans2010partial}~\cite{friedman}~\cite{lady}~\cite{ildefonso}, particularly, for general bounded domains (see, e.g., ~\cite{lieberman1992,lions2,lionscontrol,lions}). For the existence and uniqueness of solution of nonlinear parabolic systems in $\mathcal{C}^{1+\alpha}$ domains with mixed boundary conditions one can see the work developed by Pao~\cite{Pao1992,Pao2007472}, where the method of lower and upper solutions is used. The existence and uniqueness for a predator-prey type model with nonlinear reaction term is proved in~\cite{balachandran} for Neumann boundary conditions. 

In this work, we carry out a mathematical analysis of a coupled system of two reaction-diffusion-advection equations completed with Danckwerts boundary conditions, which models the interaction between a substrate and a biomass, whose concentrations are denoted by $S$ and $B$, respectively. We prove the existence and uniqueness of (weak) solutions, together with results about the nonnegativeness and boundedness of the solution. The reaction term is assumed nonlinear in $S$. The domain into consideration is a three-dimensional cylindrical bioreactor with Lipschitz boundary. The bioreactor is fed with a substrate concentration $S_{\rm e}$ at flow rate $Q$, and the treated outflow leaves the bioreactor with the same flow rate $Q$. In contrast to the models presented in~\cite{Ballyk}, \cite{Drame2004} and \cite{Qiu2004465}, we allow variable $Q$ to vary with time and space, we also allow $S_{\rm e}$ to vary with time and we consider a three-dimensional domain with Lipschitz boundary. 

This papers is organized as follows: Section \ref{modeling} introduces the mathematical model which describes the behavior of the continuous flow bioreactor and considers nonlinear reaction between the biomass and the substrate. We also state the definition of weak solution. In Section \ref{existenceuniqueness}, we first prove the existence and uniqueness of solution of a simplified linear system through some classical results for linear parabolic boundary value problems. Then, we prove the existence and uniqueness of solution of the nonlinear system applying the Schauder Fixed Point Theorem. 

\section{Mathematical modeling and weak solutions}\label{modeling}
We consider a cylindrical bioreactor as the one showed in Figure \ref{fig:bioreactor}. We denote by $\Omega \subset \mathds{R}^{3}$ its spatial domain, by $\delta \Omega =\Gamma$ its boundary and by $\bar{\Omega}$ their union, i.e, $\bar{\Omega}=\Omega \cup \delta \Omega$. We assume that $\Gamma_{\rm in}$ is the inlet upper boundary,  $\Gamma_{\rm out}$ is the outlet lower boundary and $\Gamma_{\rm wall}=\Gamma\setminus (\Gamma_{\rm in}\cup\Gamma_{\rm out})$.
\begin{figure}[ht!]
\centering
\includegraphics[width=39mm]{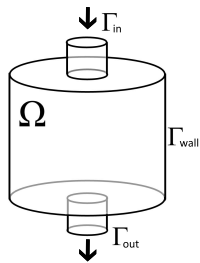}
\caption{Typical domain representation of the bioreactor geometry.}
\label{fig:bioreactor}
\end{figure}
At the beginning of the process, there is a certain amount of biomass and substrate inside $\Omega$. Furthermore, during the studied time interval,  diluted substrate enters the device through the inlet $\Gamma_{\rm in}$ and the fluid exits the bioreactor through the outlet $\Gamma_{\rm out}$.
We consider the following system describing the behavior of this particular bioreactor
\begin{equation}\label{advdifreac}
\left \{ \begin{array}{l r} 
S_{t} = {\rm div}( D_{\rm S} \nabla S - \mathbf{Q} S) - \mu(S)B  & \text{in  }\Omega \times (0,T),
\\
\\
B_{t} = {\rm div}( D_{\rm B} \nabla B - \mathbf{Q}B) + \mu(S)B  & \text{in  }\Omega \times (0,T),
\\
\\
S(x,0)=S_{\rm init}(x) &   \text{in  }\Omega,\\
\\
B(x,0)=B_{\rm init}(x) &   \text{in  }\Omega,\\
\\
\mathbf{n} \cdot( D_{\rm S}\nabla S - \mathbf{Q}S)= S_{\rm e}Q & \text{in  }\Gamma_{\rm in} \times (0,T),
\\
\\
\mathbf{n} \cdot( D_{\rm B}\nabla B - \mathbf{Q}B)= 0 & \text{in  }\Gamma_{\rm in} \cup  \Gamma_{\rm wall} \times (0,T),
\\
\\
\mathbf{n} \cdot( D_{\rm S} \nabla S - \mathbf{Q}S)= 0 & \text{in  }\Gamma_{\rm wall} \times (0,T),
\\
\\
\mathbf{n} \cdot( D_{\rm S} \nabla S) = 0 & \text{in  }\Gamma_{\rm out} \times (0,T),
\\
\\
\mathbf{n} \cdot( D_{\rm B} \nabla B) = 0 & \text{in  }\Gamma_{\rm out} \times (0,T),
\\
\\
\end{array} \right.
\end{equation}
where $T>0$ (s) is the length of the time interval for which we want to model the process, $S$ (mol/m$^{3}$) and $B$ (mol/m$^{3}$) are the substrate and biomass concentration inside the bioreactor, which diffuse throughout the water in the vessel with diffusion coefficients $D_{\rm S}$ (m$^{2}$/s) and $D_{\rm B}$ (m$^{2}$/s), respectively. The fluid flow is taken as $\mathbf{Q}=(0,0,-Q(x,t))$ where $Q$ (m/s) is the flow rate. $S_{\rm e}(t)$ (mol/m$^{3}$) is the concentration of substrate that enters into the bioreactor at time $t$ (s), $S_{\rm init}$ (mol/m$^3$) and $B_{\rm init}$ (mol/m$^{3}$) are the concentration of substrate and biomass inside the bioreactor at the beginning of the process, respectively, and $\mathbf{n}$ is the outward unit normal vector on the boundary of the domain $\Omega$. Notice that besides the Advection-Diffusion terms, we also have a term corresponding to the reaction of biomass and substrate, governed by the growth rate function $\mu(\cdot)$ (s$^{-1}$).

Now, we are interested in defining the concept of weak solution for System (\ref{advdifreac}). To do so, assuming $S,B \in W(0,T,H^{1}(\Omega),(H^{1}(\Omega))')$ (see the definition of this set in the appendix), $Q \in L^{\infty}(0,T,\mathcal{C}(\bar{\Omega}))$, $S_{\rm e}\in L^{2}(0,T)$ and $\mu \in L^{\infty}(\mathds{R})$, if we multiply the first equation of (\ref{advdifreac}) by $v\in H^{1}(\Omega)$, it follows that\\
\\
$<S_{t},v>_{(H^{1}(\Omega))' \times H^{1}(\Omega)} - <{\rm div}(D_{\rm S}\nabla S - \mathbf{Q}S), v>_{(H^{1}(\Omega))' \times H^{1}(\Omega)} + \int_{\Omega} \mu(S(x,t)) B(x,t) v(x) {\rm d}x= 0.$\\
\\
Then, applying the Green's Formula and taking into account the boundary conditions, we obtain\\
\begin{equation*}
\begin{array}{l}
<S_{t},v>_{(H^{1}(\Omega))' \times H^{1}(\Omega)}  + \int_{\Omega}  \mu(S(x,t)) B(x,t) v(x) {\rm d}x - \int_{\Gamma_{\rm in}} Q(x,t)S_{\rm e}(t) v(x) {\rm d}\Gamma_{\rm in}\\
\\ 
+ \int_{\Omega} (D_{\rm S}\nabla S(x,t) - \mathbf{Q}(x,t) S(x,t)) \nabla v(x) {\rm d}x  + \int_{\Gamma_{\rm out}}Q(x,t)S(x,t) v(x) {\rm d}\Gamma_{\rm out}= 0.\\
\end{array}
\end{equation*}
Similarly, multiplying the second equation of (\ref{advdifreac}) by $w\in H^{1}(\Omega)$, applying the Green's Formula and taking into account the boundary conditions, one has that\\
\\
$<B_{t},w>_{(H^{1}(\Omega))' \times H^{1}(\Omega)} + \int_{\Omega} (D_{\rm B}\nabla B(x,t) - \mathbf{Q}(x,t) B(x,t)) \nabla w(x) {\rm d}x + \int_{\Gamma_{\rm out}}Q(x,t)B(x,t) w(x) {\rm d}\Gamma_{\rm out}$ \\
\\
$\hspace*{1cm}  - \int_{\Omega} \mu(S(x,t))B(x,t)w(x) {\rm d}x= 0$.\\
\\ 
Let us denote $\boldsymbol{\psi}=\left ( \begin{array}{l} p \\ q \end{array} \right )$, $\boldsymbol{\phi}=\left ( \begin{array}{l} v \\ w \end{array} \right )$, $\mathbf{H}^{1}(\Omega)=H^{1}(\Omega) \times H^{1}(\Omega)$ and $(\mathbf{H}^{1}(\Omega))'= (H^{1}(\Omega))' \times (H^{1}(\Omega))'$ and consider the bilinear form $A(t,\cdot,\cdot): \mathbf{H}^{1} \times \mathbf{H}^{1} \rightarrow \mathds{R}$ defined by:\\
\\
$\begin{array}{l l} A(t,\boldsymbol{\psi},\boldsymbol{\phi})= & \int_{\Omega} (D_{\rm S}\nabla p(x) - \mathbf{Q}(x,t) p(x)) \nabla v(x) {\rm d}x + \int_{\Omega} (D_{\rm B}\nabla q(x) - \mathbf{Q}(x,t) q(x)) \nabla w(x) {\rm d}x \\
\\
& + \int_{\Gamma_{\rm out}}Q(x,t)\big(p(x) v(x)+q(x)v(x)\big) {\rm d}\Gamma_{\rm out}.
\end{array}$
\begin{defn}
A weak solution of problem (\ref{advdifreac}) is a function 
$\mathbf{u}=(S,B)$ such that \\
$S,B \in W(0,T,H^{1}(\Omega),(H^{1}(\Omega))')$ and satisfy
\begin{equation}
\left \{
\begin{array}{l}
<\mathbf{u}_{t}(\cdot),\boldsymbol{\phi}>_{(\mathbf{H}^{1}(\Omega))'\times \mathbf{H}^{1}(\Omega)}  + A(\cdot,\mathbf{u}(\cdot),\boldsymbol{\phi})=  \\
\\
\hspace*{5mm} \int_{\Gamma_{\rm in}} Q(x,\cdot)S_{\rm e}(\cdot)v(x) {\rm d}\Gamma_{\rm in} + \int_{\Omega} \mu(S(x,\cdot))B(x,\cdot) (w(x)-v(x)){\rm d}x \\
\\
\hspace*{65mm}\text{for all } \boldsymbol{\phi}=(v,w) \in \mathbf{H}^{1}(\Omega)\\
\end{array} \right.
\end{equation}
in the sense of $\mathcal{D}'(0,T)$ (see, e.g., ~\cite{gelfand1964}), i.e., all the terms above are considered as distributions in $t$. Notice that 
\\
\\
$\begin{array}{l l}<\mathbf{u}_{t}(\cdot),\boldsymbol{\phi}>_{(\mathbf{H}^{1}(\Omega))'\times \mathbf{H}^{1}(\Omega)}& = <S_{t}(\cdot),v>_{(H^{1}(\Omega))'\times H^{1}(\Omega)} + <B_{t}(\cdot),w>_{(H^{1}(\Omega))'\times H^{1}(\Omega)} \\
\\
& = \frac{{\rm d}}{{\rm d} t} \Big ( \int_{\Omega} S(\cdot,x)v(x){\rm d} x + \int_{\Omega} B(\cdot,x)w(x) {\rm d} x \Big ) \end{array}$\\
in the sense of $\mathcal{D}'(0,T)$.
\end{defn} 
\section{Existence, uniqueness, nonnegativity and boundedness of the solution}\label{existenceuniqueness}
We are first interested in proving the following result:
\begin{thm}[Existence of solution]\label{existencianolineal}
Let us assume that $Q\in L^{\infty}(0,T,\mathcal{C}(\bar{\Omega}))$ is nonnegative, $S_{\rm e}\in L^{2}(0,T)$, $S_{\rm init},B_{\rm init} \in L^{2}(\Omega)$, $D_{\rm S}, D_{\rm B} >0$ and $\mu \in L^{\infty}(\mathds{R})$ is continuous. Then, System (\ref{advdifreac}) has at least one weak solution $(S,B)$.
\end{thm}
\begin{rmk}
Notice that we assume that $Q$ is nonnegative because of its physical meaning. However, in order to prove Theorem \ref{existencianolineal} it suffices to consider $Q$ with negative part $Q^{-}$ such that $\|Q^{-}\|_{L^{\infty}(\bar{\Omega}\times(0,T))} <\frac{\min(D_{\rm S},D_{\rm B})}{C_{\rm T}^{2}}$ a.e. $t \in (0,T)$, where $C_{\rm T}$ is the constant coming from the Trace Theorem (see e.g., \cite{angel}).   
\end{rmk}
In order to prove Theorem \ref{existencianolineal}, we first investigate the existence and uniqueness of solution of the following linear parabolic system:
\begin{equation}\label{generico}
\left \{ \begin{array}{l r}
S_{t} - {\rm div}( D_{\rm S} \nabla S - \mathbf{Q}S) + cB = 0 & \text{in  }\Omega \times (0,T),
\\
\\
B_{t} - {\rm div}( D_{\rm B} \nabla B - \mathbf{Q}B) - cB = 0 & \text{in  }\Omega \times (0,T),
\\
\\
S(x,0)=S_{\rm init}(x) &   \text{in  }\Omega,\\
\\
B(x,0)=B_{\rm init}(x) &   \text{in  }\Omega,\\
\\
\mathbf{n} \cdot( D_{\rm S}\nabla S - \mathbf{Q}S)= S_{\rm e}Q & \text{in  }\Gamma_{\rm in} \times (0,T),
\\
\\
\mathbf{n} \cdot( D_{\rm B}\nabla B - \mathbf{Q}B)= 0 & \text{in  }\Gamma_{\rm in} \cup  \Gamma_{\rm wall} \times (0,T),
\\
\\
\mathbf{n} \cdot( D_{\rm S} \nabla S - \mathbf{Q}S)= 0 & \text{in  }\Gamma_{\rm wall} \times (0,T),
\\
\\
\mathbf{n} \cdot( D_{\rm S} \nabla S) = 0 & \text{in  }\Gamma_{\rm out} \times (0,T),
\\
\\
\mathbf{n} \cdot( D_{\rm B} \nabla B) = 0 & \text{in  }\Gamma_{\rm out} \times (0,T),
\\
\\
\end{array} \right.
\end{equation}
where $c \in L^{\infty}(\Omega \times (0,T))$. Proceeding analogously to the nonlinear case, we first define the concept of weak solution for this system. 
\begin{defn}
A weak solution of problem (\ref{generico}) is a function $\mathbf{u}=(S,B)$ such that \\
$S,B \in W(0,T,H^{1}(\Omega),(H^{1}(\Omega))')$ and satisfy
\begin{equation}\left \{
\begin{array}{r}
<\mathbf{u}_{t}(\cdot),\boldsymbol{\phi}>_{(	\mathbf{H}^{1}(\Omega))' \times \mathbf{H}^{1}(\Omega)} + \bar{A}(\cdot,\mathbf{u}(\cdot),\boldsymbol{\phi}) = \int_{\Gamma_{\rm in}} Q(x,\cdot)S_{\rm e}(\cdot)v(x) {\rm d}\Gamma_{\rm in} \\
\\
\text{for all } \phi=(v,w) \in \mathbf{H}^{1}(\Omega))\\
\end{array} \right.
\end{equation}
in the sense of $\mathcal{D}'(0,T)$. Here, $\boldsymbol{\psi}=\left ( \begin{array}{l} p \\ q \end{array} \right )$, $\boldsymbol{\phi}=\left ( \begin{array}{l} v \\ w \end{array} \right )$ and the bilinear form $\bar{A}(t,\cdot,\cdot): \mathbf{H}^{1} \times \mathbf{H}^{1} \rightarrow \mathds{R}$ is defined by\\
\\
$\begin{array}{l l} \bar{A}(t,\boldsymbol{\psi},\boldsymbol{\phi}) = & \int_{\Omega} (D_{\rm S}\nabla p(x) - \mathbf{Q}(x,t) p(x)) \nabla v(x) {\rm d}x  + \int_{\Omega} (D_{\rm B}\nabla q(x) - \mathbf{Q}(x,t) q(x)) \nabla w(x) {\rm d}x \\
\\
& + \int_{\Omega} c(x,t) q(x) \big(v(x)-w(x)\big){\rm d}x  + \int_{\Gamma_{\rm out}}Q(x,t)\big(p(x) v(x)+q(x)w(x)\big) {\rm d}\Gamma_{\rm out}. 
\end{array}$
\end{defn}
We now focus on proving the existence and uniqueness of solution of the linear system:
\begin{thm}\label{unicidadlineal}
Under the assumptions of Theorem \ref{existencianolineal}, problem (\ref{generico}) has a unique weak solution $(S,B)$.
\end{thm}
\begin{proof}[Proof of Theorem \ref{unicidadlineal}]
Since the equation for $B$ in System (\ref{generico}) does not depend on $S$, Theorem \ref{unicidadlineal} can be proved by applying Theorem \ref{teoremalions} twice (taking $V=H^{1}(\Omega)$ and $H=L^{2}(\Omega)$) to a single equation, namely
\begin{equation}\label{singlequation}
\left \{
\begin{array}{l r}
R_t =  {\rm div} ( D_{\rm R} \nabla R - \mathbf{Q} R ) + \gamma R - f & \text{in }\Omega\times (0,T), \\
\\
R(x,0) = R_{\rm init}(x) & \text{in } \Omega , \\
\\
\mathbf{n} \cdot ( D_{\rm R} \nabla R - \mathbf{Q} R )  =  G & \text{in } \Gamma_{\rm in}\times (0,T), \\
\\
\mathbf{n} \cdot ( D_{\rm R} \nabla R - \mathbf{Q} R )  = 0 & \text{in } \Gamma_{\rm wall} \times (0,T), \\
\\
\mathbf{n} \cdot ( D_{\rm R} \nabla R) =0 & \text{ in } \Gamma_{\rm out}\times (0,T),\\
\end{array}
\right. 
\end{equation}
with $(D_{\rm R},\gamma,f,R_{\rm init},G)=(D_{\rm B},c,0,B_{\rm init}, 0)$, and then with $(D_{\rm R},\gamma,f,R_{\rm init},G)=(D_{\rm S},0,cB,S_{\rm init},S_e Q)$.
\\
\\
Notice that the application of Theorem \ref{teoremalions} is analogous in both cases. In order to shorten the length of this work we only present a detailed proof for the second one, assuming that $B \in L^{2}(\Omega\times(0,T))$.
We define the bilinear operator $\bar{a}(t,\cdot,\cdot):H^{1}(\Omega)\times H^{1}(\Omega) \rightarrow \mathds{R}$ as  \\
\\
$\bar{a}(t,p,v)= \int_{\Omega} (D_{\rm S}\nabla p(x) - \mathbf{Q}(x,t) p(x)) \nabla v(x) {\rm d}x + \int_{\Gamma_{\rm out}}Q(x,t)p(x) v(x) {\rm d}\Gamma_{\rm out}$  \\
\\
so that a weak solution of equation (\ref{singlequation}) is a function 
$S \in W(0,T,H^{1}(\Omega),(H^{1}(\Omega))')$ satisfying
\begin{equation}\left \{
\begin{array}{l}
<S_{t}(\cdot),\boldsymbol{\phi}>_{(H^{1}(\Omega))' \times H^{1}(\Omega)} + \bar{a}(\cdot,S(\cdot),v) = \int_{\Gamma_{\rm in}} Q(x,\cdot)S_{\rm e}(\cdot)v(x) {\rm d}\Gamma_{\rm in}  - \int_{\Omega} c(x,\cdot)B(x,\cdot)v(x){\rm d} x  \\
\\
\hspace*{65mm}\text{for all } v \in H^{1}(\Omega))\\
\end{array} \right.
\end{equation}
in the sense of $\mathcal{D}'(0,T)$.  
\\
\\
Let us see that $\bar{a}$ satisfies condition (\ref{hipotesisauv1}):\\
For all $p, v \in H^{1}(\Omega)$, function $t \rightarrow \bar{A}(t,p,v)$ is Lebesgue measurable. This follows from the fact that $Q$ is assumed to be Lebesgue measurable function.\\
\\
To be able to apply Theorem \ref{teoremalions}, we need to find $k\in \mathds{R}$ such that $|\bar{a}(t,p,v)|\leq k \|p\|_{H^{1}(\Omega)} \|v\|_{H^{1}(\Omega)} $ for all $p,v \in H^{1}(\Omega)$, $a.e.t\in (0,T)$. Now,
\\
\\
$\begin{array}{l l}|\bar{a}(t,p,v)| \leq & D_{\rm S} \int_{\Omega}|\nabla p(x)| |\nabla v(x)|{\rm d}x + \|Q\|_{L^{\infty}(\bar{\Omega}\times(0,T))}  \int_{\Omega}|p(x)||\nabla v(x)| {\rm d}x\\
\\
&+ \|Q\|_{L^{\infty}(\bar{\Omega}\times(0,T))} \|p\|_{L^{2}(\Gamma_{\rm out})} \|v\|_{L^{2}(\Gamma_{\rm out})}.\end{array}$\\
\\
Then, using the Trace Theorem (see e.g., \cite{angel}), we can conclude that there exist a constant $C_{\rm T}>0$ such that 
\\
\\
$|\bar{a}(t,p,v)| \leq ( D_{\rm S} + (1+C_{\rm T}^{2})\|Q\|_{L^{\infty}(\bar{\Omega}\times(0,T))} \|p\|_{H^{1}(\Omega)} \|v\|_{H^{1}(\Omega)}. $\\
\\
Let us see that $\bar{a}$ satisfies condition (\ref{hipotesisauv2}):\\
We need to find $\alpha,\lambda>0$ such that $\bar{a}(t,p,p) + \lambda  \|p\|_{L^{2}(\Omega)}^{2}\geq \alpha \|p\|_{H^{1}(\Omega)}^{2}$ for all $p \in H^{1}(\Omega)$, a.e. $t\in (0,T)$. We have that
\\
\\
$\bar{a}(t,p,p) =  D_{\rm S} \| \nabla p \|_{L^{2}(\Omega)}^{2} - \int_{\Omega} \mathbf{Q}(x,t) p(x) \nabla p(x) {\rm d}x + \int_{\Gamma_{\rm out}} Q(x,t)p(x)^{2}{\rm d} x.$\\
\\
Applying Young's inequality (see e.g., \cite{evans2010partial}) with $\epsilon>0$, to be chosen later, the following inequality holds:\\
\\
$- \int_{\Omega} \mathbf{Q}(x,t) p(x) \nabla p(x) {\rm d}x \geq -( \epsilon \|p\|_{L^{2}(\Omega)}^{2} + \frac{1}{4\epsilon}\|\nabla p\|_{L^{2}(\Omega)}^{2} )\|Q\|_{L^{\infty}(\bar{\Omega} \times (0,T))} $.\\
\\
Furthermore,
\\
\\
$\int_{\Gamma_{\rm out}} Q(x,t)p(x)^{2}{\rm d} \Gamma_{\rm out} \geq 0$, since $Q$ is nonnegative by assumption.\\
\\
Consequently, \\
\\
$\bar{a}(t,p,p) \geq ( D_{\rm S} - \frac{1}{4\epsilon} \|Q\|_{L^{\infty}(\bar{\Omega} \times (0,T))}) \| \nabla p\|_{L^{2}(\Omega)}^{2} - \epsilon\|Q\|_{L^{\infty}(\bar{\Omega} \times (0,T))})\|p\|_{L^{2}(\Omega)}^{2}.$\\
\\
We choose $\epsilon>0$ such that $$ \alpha_{1}= D_{\rm S} - \frac{1}{4\epsilon} \|Q\|_{L^{\infty}(\bar{\Omega} \times (0,T))}>0,$$ 
and then, we choose $\lambda>0$ such that 
$$\alpha_{2}= \lambda- \epsilon\|Q\|_{L^{\infty}(\bar{\Omega} \times (0,T))} >0.$$ 
Therefore, choosing $\alpha= \min\{ \alpha_{1},\alpha_{2}\}$, one has that 
$$|\bar{a}(t,p,p)| + \lambda \|p\|_{L^{2}(\Omega)}^{2} \geq \alpha \|p\|_{H^{1}(\Omega)}^{2}.$$
Finally, in order to apply Theorem \ref{teoremalions} we need to prove that $f: (0,T) \longrightarrow (H^{1}(\Omega))'$, with $f(t): H^{1}(\Omega) \longrightarrow \mathds{R}$ defined by
$$v \rightarrow \int_{\Gamma_{\rm in}} Q(x,t)S_{\rm e}(t) v(x) {\rm d} \Gamma_{\rm in} + \int_{\Omega}c(x,t)B(x,t)v(x){\rm d} x,$$
is in $L^{2}(0,T,H^{1}(\Omega)')$. \\
Firstly, we must see that $f(t)$ is linear and continuous a.e.  $t\in (0,T)$. The linearity of $f(t)$ follows from the linearity of the integral. Because of this linearity, the continuity property is equivalent to the existence of $k(t)>0$ such that $|f(t)(v)|\leq k(t) \|v\|_{H^{1}(\Omega)}$, $\forall v \in H^{1}(\Omega)$. 
But one has\\
\\
$\begin{array}{r l}|f(t)(v)| = &   |\int_{\Gamma_{\rm in}} Q(x,t)S_{\rm e}(t) v(x) {\rm d} \Gamma_{\rm in} + \int_{\Omega}c(x,t)B(x,t)v(x){\rm d} x |\\
\\
\leq & \|Q(\cdot,t)\|_{L^{\infty}(\bar{\Omega})} |S_{\rm e}(t)| |\Gamma_{\rm in}|^{\frac{1}{2}} \|v\|_{L^{2}(\Gamma_{\rm in})} + \|c(\cdot,t)\|_{L^{\infty}(\Omega)} \|B(\cdot,t)\|_{L^{2}(\Omega)} \|v\|_{L^{2}(\Omega)} \end{array}$ \\
\\
a.e. $t \in (0,T)$, where $|\Gamma_{\rm in}|$ is the Lebesgue measure of $\Gamma_{\rm in}$. Using the Trace Theorem (see e.g., \cite{angel}), we conclude that there exists a constant $C_{\rm T}>0$ such that:
$$\begin{array}{l}|f(t)(v)| \leq C_{\rm T} |\Gamma_{\rm in}|^{\frac{1}{2}}\|Q(\cdot,t)\|_{L^{\infty}(\bar{\Omega})}|S_{\rm e}(t)|\|v\|_{H^{1}(\Omega)} + \|c(\cdot,t)\|_{L^{\infty}(\Omega)} \|B(\cdot,t)\|_{L^{2}(\Omega)} \|v\|_{L^{2}(\Omega)}= k(t)\|v\|_{H^{1}(\Omega)}, \end{array}$$
with $k(t)=  C_{\rm T} |\Gamma_{\rm in}|^{\frac{1}{2}}\|Q(\cdot,t)\|_{L^{\infty}(\bar{\Omega})}|S_{\rm e}(t)| + \|c(\cdot,t)\|_{L^{\infty}(\Omega)} \|B(\cdot,t)\|_{L^{2}(\Omega)}$. \\
Secondly, we must see that $\int_{0}^{T} \|f(t)\|_{(H^{1}(\Omega))'}^{2}{\rm d} t < \infty$.  We use that
$$\|G\|_{(H^{1}(\Omega))'}= \sup_{ \mycom{v \in H^{1}(\Omega)}{\|v\|\leq 1}} |<G,v>|,$$
and thus, by the hypothesis on $Q$, $S_{\rm e}$, $c$ and $B$ we have that\\
\\
$\begin{array}{l l}
\int_{0}^{T} \|f(t)\|_{(H^{1}(\Omega))'}^{2}{\rm d} t & \leq \int_{0}^{T} \big(|\Gamma_{\rm in}|^{\frac{1}{2}}C_{\rm T}|S_{\rm e}(t)| \|Q(\cdot,t)\|_{L^{\infty}(\bar{\Omega})} + \|c(\cdot,t)\|_{L^{\infty}(\Omega)} \|B(\cdot,t)\|_{L^{2}(\Omega)} \big)^{2} {\rm d} t \\
\\
& \leq |\Gamma_{\rm in}|C_{\rm T}^{2} \|Q\|_{L^{\infty}(\bar{\Omega}\times(0,T))}^{2} \|S_{\rm e}\|_{L^{2}(0,T)}^{2}  +\|c\|_{L^{\infty}(\Omega\times(0,T))}^{2} \|B\|_{L^{2}(\Omega\times(0,T))}^{2} \\
\\
& \hspace*{5mm} +2|\Gamma_{\rm in}|^{\frac{1}{2}}C_{\rm T}\|Q\|_{L^{\infty}(\bar{\Omega}\times(0,T))}\|c\|_{L^{\infty}(\Omega\times(0,T))} \|S_{\rm e}\|_{L^{2}(0,T)} \|B\|_{L^{2}(\Omega\times(0,T))} < \infty.
\end{array}$\\
\\
Since we have proved that all the assumptions of Theorem \ref{teoremalions} are satisfied, the proof of Theorem \ref{unicidadlineal} is finished.
\end{proof}
Before proving  Theorem \ref{existencianolineal}, we prove the following result:
\begin{prop}\label{solucionacotada}
If $(S,B)$ is the weak solution of System (\ref{generico}), then
$$\|S\|_{W(0,T,H^{1}(\Omega),(H^{1}(\Omega))')} \leq C \text{ and } \|B\|_{W(0,T,H^{1}(\Omega),(H^{1}(\Omega))')} \leq C,$$
where $C$ depends on $D_{\rm S}$, $D_{\rm B}$, $\|S_{\rm init}\|_{L^{2}(\Omega)}$, $\|B_{\rm init}\|_{L^{2}(\Omega)}$, $\|Q\|_{L^{\infty}(\bar{\Omega}\times(0,T))}$, $\|S_{\rm e}\|_{L^{2}(0,T)}$, $|\Gamma_{\rm in}|$, $T$, $\|c\|_{L^{\infty}(\Omega \times (0,T))}$ and $C_{\rm T}$ (i.e., the constant coming from the Trace Theorem (see e.g., \cite{angel}).
\end{prop}

\begin{proof}[Proof of Proposition \ref{solucionacotada}]
From the first equation in System (\ref{generico}), it follows that \\
\\
$\begin{array}{l l}
\|\frac{{\rm d} S}{{\rm d} t}\|_{L^{2}(0,T,(H^{1}(\Omega))')} &= \sup_{ \mycom{\phi \in L^{2}(0,T,H^{1}(\Omega)),}{\|\phi\|\leq 1} } |<\frac{{\rm d} S}{{\rm d} t},\phi>| 
\\
\\
&= \sup_{ \mycom{\phi \in L^{2}(0,T,H^{1}(\Omega)),}{\|\phi\|\leq 1 }}|<{\rm div}(D_{\rm S}\nabla S - \mathbf{Q}S) - cB,\phi>|.
\end{array}$\\
\\
If $C_{\rm T}$ is the constant coming from the Trace Theorem (see e.g., \cite{angel}), one has that
\begin{equation}\label{desigualdaderivada}
\begin{array}{l l}
\|\frac{{\rm d} S}{{\rm d} t}\|_{L^{2}(0,T,(H^{1}(\Omega))')} \leq & C_{\rm T} \|Q\|_{L^{\infty}(\bar{\Omega}\times(0,T))} \|S_{\rm e}\|_{L^{2}(0,T)} |\Gamma_{\rm in}|^{\frac{1}{2}} + \|c\|_{L^{\infty}(\Omega\times(0,T))} \|B\|_{L^{2}(0,T,L^{2}(\Omega))} \\
\\
& +(D_{\rm S}+(1+C_{\rm T}^{2})\|Q\|_{L^{\infty}(\bar{\Omega}\times(0,T))})\|S\|_{L^{2}(0,T,H^{1}(\Omega))}.
\end{array}
\end{equation}
Similarly, from the second equation in System (\ref{generico}), it follows that
\begin{equation}\label{desigualdaderivadab}
\|\frac{{\rm d} B}{{\rm d} t}\|_{L^{2}(0,T,(H^{1}(\Omega))')} \leq (D_{\rm B}+ (1+C_{\rm T}^{2})\|Q\|_{L^{\infty}(\bar{\Omega}\times(0,T))}+\|c\|_{L^{\infty}(\Omega\times(0,T))}) \|B\|_{L^{2}(0,T,H^{1}(\Omega))}.
\end{equation}
Now, in order to obtain an estimate for $\|S\|_{L^{2}(0,T,H^{1}(\Omega))}$, we consider $\lambda \geq 0$ and the variable $\bar{S}=e^{-\lambda t}S$, that fulfill
\begin{equation}\label{ecuacionparacotacion}
\bar{S}_{t}+\lambda \bar{S} - {\rm div}(D_{\rm S}\nabla \bar{S}-\mathbf{Q}\bar{S})+c\bar{B}=0\\
\end{equation}
Multiplying (\ref{ecuacionparacotacion}) by $\bar{S}$ (here, this multiplication is in the sense of the duality product $<\cdot,\cdot>_{(H^{1}(\Omega))'\times H^{1}(\Omega)}$) and integrating, one obtains
\begin{equation}\label{variacionalsub}
\begin{array}{l}
\frac{1}{2} \|\bar{S}(T)\|_{L^{2}(\Omega)}^{2} + \lambda \int_{0}^{T} \|\bar{S}(\tau)\|_{L^{2}(\Omega)}^{2}{\rm d} \tau + \int_{0}^{T} \int_{\Gamma_{\rm out}} Q(x,\tau) \bar{S}^{2}(x,\tau){\rm d} x {\rm d} \tau \\
\\
+ D_{\rm S} \int_{0}^{T} \|\nabla \bar{S}(\tau) \|_{L^{2}(\Omega)}^{2} {\rm d} \tau =  \frac{1}{2} \|S_{\rm init}\|_{L^{2}(\Omega)}^{2} + \underbrace{\int_{0}^{T} e^{-\lambda \tau} \int_{\Gamma_{\rm in}} Q(x,\tau)S_{\rm e}(\tau)\bar{S}(x,\tau){\rm d} x {\rm d} \tau}_{(\ast)} \\
\\
+ \underbrace{\int_{0}^{T} \int_{\Omega} \mathbf{Q}(x,\tau) \bar{S}(x,\tau) \nabla \bar{S}(x,\tau) {\rm d}x {\rm d} \tau}_{(\ast \ast)} - \underbrace{\int_{0}^{T} \int_{\Omega} c(x,\tau)\bar{S}(x,\tau)\bar{B}(x,\tau) {\rm d} x {\rm d} \tau}_{(\ast \ast \ast)}.\\
\end{array}
\end{equation}
Applying Young's Inequality (see e.g., \cite{evans2010partial}) in ($\ast$), ($\ast \ast$) and ($\ast \ast \ast$) (with $\epsilon_{1}>0$, $\epsilon_{2}>0$ and $\epsilon=\frac{1}{2}$, respectively) and the Trace Theorem (see e.g., \cite{angel}) in ($\ast$), it follows
\begin{equation}\label{variacionalsub2}
\begin{array}{l}
\frac{1}{2} \|\bar{S}(T)\|_{L^{2}(\Omega)}^{2}  + (D_{\rm S}- \|Q\|_{L^{\infty}(\bar{\Omega}\times(0,T))}(\frac{1}{4\epsilon_{2}} + \frac{|\Gamma_{\rm in}|^{\frac{1}{2}} C_{\rm T}^{2}}{4\epsilon_{1}}))\|\nabla \bar{S}(\tau) \|_{L^{2}(0,T,L^{2}(\Omega))}^{2} \\
\\  
+(\lambda - \|Q\|_{L^{\infty}(\bar{\Omega}\times(0,T))}(\epsilon_{2} + \frac{|\Gamma_{\rm in}|^{\frac{1}{2}} C_{\rm T}^{2}}{4\epsilon_{1}} )-\frac{\|c\|_{L^{\infty}(\Omega \times (0,T))}}{2}) \|\bar{S}\|_{L^{2}(0,T,L^{2}(\Omega))}^{2}\\
\\
\leq  \frac{1}{2} \|S_{\rm init}\|_{L^{2}(\Omega)}^{2} + \epsilon_{1}\|Q\|_{L^{\infty}(\bar{\Omega}\times(0,T))} \|S_{\rm e}\|_{L^{2}(0,T)}^{2}|\Gamma_{\rm in}|^{\frac{1}{2}} + \frac{\|c\|_{L^{\infty}(\Omega \times (0,T))} \|\bar{B}\|_{L^{2}(0,T,L^{2}(\Omega))}^{2}}{2}.
\end{array}
\end{equation}
Considering the variable $\bar{B}=e^{-\lambda t}B$ and using the same reasoning as the one followed above, one has that 
\begin{equation}\label{variacionalb}
\begin{array}{l}
\frac{1}{2} \|\bar{B}(T)\|_{L^{2}(\Omega)}^{2} +  (\lambda - \epsilon_{3}\|Q\|_{L^{\infty}(\bar{\Omega}\times(0,T))} -\|c\|_{L^{\infty}(\Omega \times (0,T))}) \|\bar{B}\|_{L^{2}(0,T,L^{2}(\Omega))}^{2} \\ \\
\hspace*{1cm} + (D_{\rm B}-\frac{\|Q\|_{L^{\infty}(\bar{\Omega}\times(0,T))}}{4\epsilon_{3}})\|\nabla \bar{B}(\tau) \|_{L^{2}(0,T,L^{2}(\Omega))}^{2} \leq  \frac{1}{2} \|B_{\rm init}\|_{L^{2}(\Omega)}^{2}.
\end{array}
\end{equation}
Choosing $\epsilon_{1}$, $\epsilon_{2}$ and $\epsilon_{3}$ such that $D_{\rm S}\geq \|Q\|_{L^{\infty}(\bar{\Omega}\times(0,T))} (\frac{1}{4\epsilon_{2}} + \frac{|\Gamma_{\rm in}|^{\frac{1}{2}} C_{\rm T}^{2}}{4\epsilon_{1}})$, $\epsilon_{3}\geq \frac{\|Q\|_{L^{\infty}(\bar{\Omega}\times(0,T))}}{4D_{\rm B}}$ and $\lambda > \|Q\|_{L^{\infty}(\bar{\Omega}\times(0,T))}\max( \epsilon_{3}, \epsilon_{2} + \frac{C_{\rm T}^{2}|\Gamma_{\rm in}|^{\frac{1}{2}}}{4\epsilon_{1}}) + \|c\|_{L^{\infty}(\Omega \times (0,T))}$, it follows that 
\begin{equation}\label{acotacionsyb1}
\begin{array}{l l}
\|\bar{B}\|_{L^{2}(0,T,H^{1}(\Omega))}^{2} \leq & \alpha_{1}\|B_{\rm init}\|_{L^{2}(\Omega)}^{2},\\
\\
\|\bar{S}\|_{L^{2}(0,T,H^{1}(\Omega))}^{2} \leq & \alpha_{2}( \frac{1}{2} \|S_{\rm init}\|_{L^{2}(\Omega)}^{2} + \epsilon_{1}\|Q\|_{L^{\infty}(\bar{\Omega}\times(0,T))} \|S_{\rm e}\|_{L^{2}(0,T)}^{2}|\Gamma_{\rm in}|^{\frac{1}{2}})\\
\\
&+ \alpha_{1}\alpha_{2}\frac{\|c\|_{L^{\infty}(\Omega \times (0,T))}}{2} \|B_{\rm init}\|_{L^{2}(\Omega)}^{2},\\ 
\end{array}
\end{equation}
where $\alpha_{1},\alpha_{2}>0$ depend on $|\Gamma_{\rm in}|, \|c\|_{L^{\infty}(\Omega \times (0,T))}, \|Q\|_{L^{\infty}(0,T)}, C_{\rm T}, D_{\rm S}$ and $D_{\rm B}$. 
\\
\\
Furthermore, it is straight forward to see that
\begin{equation}\label{acotacionsyb2}
\begin{array}{l}
\|B\|_{L^{2}(0,T,L^{2}(\Omega))}^{2} \leq e^{2\lambda T} \|\bar{B}\|_{L^{2}(0,T,L^{2}(\Omega))}^{2} ,\\
\\
\|S\|_{L^{2}(0,T,H^{1}(\Omega))}^{2} \leq e^{2\lambda T} \|\bar{S}\|_{L^{2}(0,T,H^{1}(\Omega))}^{2}. \\ 
\end{array}
\end{equation}
From (\ref{desigualdaderivada}), (\ref{acotacionsyb1}) and (\ref{acotacionsyb2}), it follows that
$$\|S\|_{W(0,T,H^{1}(\Omega),(H^{1}(\Omega))')}, \|B\|_{W(0,T,H^{1}(\Omega),(H^{1}(\Omega))')} \leq C,$$
where $C$ depends on $T, \|S_{\rm init}\|_{L^{2}(\Omega)},\|B_{\rm init}\|_{L^{2}(\Omega)},D_{\rm S},D_{\rm B},\|Q\|_{L^{\infty}(\bar{\Omega}\times(0,T))},$ \\
$\|S_{\rm e}\|_{L^{2}(0,T)}$, $\|c\|_{L^{\infty}(\Omega \times (0,T))}, |\Gamma_{\rm in}|$ and $C_{\rm T}$.
\end{proof}
\begin{proof}[Proof of Theorem \ref{existencianolineal}]
In order to prove the existence of solution, we apply Schauder Fixed Point Theorem (see e.g., \cite{evans2010partial}). We have to choose a Banach space $X$ and a compact and convex subset $K\subset X$.

We consider the Banach Space $W(0,T,H^{1}(\Omega),(H^{1}(\Omega))')$, which is compactly embedded in $L^{2}(0,T,L^{2}(\Omega))$ (see Lemma \ref{aubin}).

If $Z\in W(0,T,H^{1}(\Omega),(H^{1}(\Omega))')$ and we solve the linear System (\ref{generico}) with $c(x,t)=\mu(Z(x,t))$, Theorem \ref{unicidadlineal} proves that there exists a unique weak solution $(S_{\rm Z},B_{\rm Z})$ with $S_{\rm Z}, B_{\rm Z} \in W(0,T,H^{1}(\Omega),(H^{1}(\Omega))')$. Furthermore, Proposition \ref{solucionacotada} shows that 
$$\|B_{Z}\|_{W(0,T,H^{1}(\Omega),(H^{1}(\Omega))')} \leq C \text{ and }\|S_{Z}\|_{W(0,T,H^{1}(\Omega),(H^{1}(0,T))')} \leq C,$$ 
where $C$ depends (among others) on the norm of $\mu(Z(x,t))$. Since $\mu(\cdot) \in L^{\infty}(\mathds{R})$ it follows that for all $Z \in W(0,T,H^{1}(\Omega),(H^{1}(\Omega))')$, we have
$$\|B_{Z}\|_{W(0,T,H^{1}(\Omega),(H^{1}(\Omega))')} \leq \bar{C} \text{ and }\|S_{Z}\|_{W(0,T,H^{1}(\Omega),(H^{1}(\Omega))')}\leq \bar{C},$$
where $\bar{C}$ is a constant depending (among others) on $\|\mu\|_{L^{\infty}(\mathds{R})}$.\\
\\
If we define the set 
\begin{equation} \label{defK}
K := \{ z \in W( 0,T,H^{1}(\Omega),(H^{1}(\Omega))'): \|z\|_{W(0,T,H^{1}(\Omega),(H^{1}(\Omega))')} \leq \bar{C}\},
\end{equation}
from Lemma \ref{aubin} and the definition of compact operator, $K$ is a compact set of the Banach Space $X:= L^{2}(0,T,L^{2}(\Omega))$.\\

Let us define the application $A: K \rightarrow K$ by $A(Z)=S_{\rm Z}$. We prove Theorem \ref{existencianolineal} by showing that $A$ has a fixed point. In order to apply Schauder Fixed Point Theorem, it is enough to prove that $A$ is continuous.

In this direction, if $\{ Z_{\rm n} \}_{n} \subset K $, $Z \in K$ are such that $\| Z_{\rm n} - Z \|_{X} \stackrel{n \rightarrow \infty}{\longrightarrow}0$, we must prove that $$\|A(Z_{\rm n})-A(Z)\|_{X} = \|S_{\rm Z_{\rm n}} - S_{\rm Z}\|_{L^{2}(0,T,L^{2}(\Omega))} \stackrel{n \rightarrow \infty}{\longrightarrow}0.$$

Let $(S_{\rm Z_{\rm n}},B_{\rm Z_{\rm n}})$ and $(S_{\rm Z},B_{\rm Z})$ be the weak solutions of linear system (\ref{generico}) when $c(x,t)= \mu(Z_{\rm n}(x,t))$ and $c(x,t)=\mu(Z(x,t))$, respectively. We denote $V_{n}=S_{\rm Z_{\rm n}}-S_{\rm Z}$ and $W_{n}=B_{\rm Z_{\rm n}}-B_{\rm Z}$. Then $(V_{n},W_{n})$ is a weak solution of:
\begin{equation*}\label{resta}
\left \{ \begin{array}{l r}
(V_{n})_{t} - {\rm div}( D_{\rm S} \nabla V_{n} - \mathbf{Q}V_{n}) + \mu(Z)B_{\rm Z} - \mu(Z_{\rm n})B_{\rm Z_{\rm n}} = 0 & \text{in  }\Omega \times (0,T),
\\
\\
(W_{n})_{t} - {\rm div}( D_{\rm B} \nabla W_{n} - \mathbf{Q}W_{n}) - \mu(Z)B_{\rm Z} + \mu(Z_{\rm n})B_{\rm Z_{\rm n}} = 0 & \text{in  }\Omega \times (0,T),
\end{array} \right.
\end{equation*}
with the initial and boundary conditions
\begin{equation*}\label{restabound}
\left \{ \begin{array}{l c r}
V_{n}(x,0)= 0 &  & \text{in  }\Omega,\\
\\
W_{n}(x,0)= 0 &  & \text{in  }\Omega,\\
\\
\mathbf{n} \cdot( D_{\rm S}\nabla V_{n} - \mathbf{Q}V_{n})= 0 & &\text{in  }\Gamma_{\rm in} \cup  \Gamma_{\rm wall} \times (0,T),
\\
\\
\mathbf{n} \cdot( D_{\rm B}\nabla W_{n} - \mathbf{Q}W_{n})= 0 & &\text{in  }\Gamma_{\rm in} \cup  \Gamma_{\rm wall} \times (0,T),
\\
\\
\mathbf{n} \cdot( D_{\rm S} \nabla V_{n}) = 0 & &\text{in  }\Gamma_{\rm out} \times (0,T),
\\
\\
\mathbf{n} \cdot( D_{\rm B} \nabla W_{n}) = 0 & &\text{in  }\Gamma_{\rm out} \times (0,T).
\end{array} \right.
\end{equation*}

Given $\lambda>0$, then $\bar{V}_{n}= e^{-\lambda t} V_{n}$ and $\bar{W}_{n}= e^{-\lambda t} W_{n}$ fulfill:
\begin{equation}\label{variacionalconlambda}
\begin{array}{l}
(\bar{V}_{n})_{\rm t} + \lambda \bar{V}_{n} - {\rm div}( D_{\rm S} \nabla \bar{V}_{n} - \mathbf{Q}\bar{V}_{n}) + e^{-\lambda t}\Big(\mu(Z)B_{\rm Z} - \mu(Z_{\rm n})B_{\rm Z_{\rm n}}\Big)=0,\\
\\
(\bar{W}_{n})_{\rm t} + \lambda \bar{W}_{n} - {\rm div}( D_{\rm B} \nabla \bar{W}_{n} - \mathbf{Q}\bar{W}_{n}) - e^{-\lambda t}\Big(\mu(Z)B_{\rm Z} - \mu(Z_{\rm n})B_{\rm Z_{\rm n}}\Big)=0.
\end{array}
\end{equation}

Multiplying the first equation of (\ref{variacionalconlambda}) by $\bar{V}_{n}$ and integrating, one obtains:
\begin{equation}\label{variacionalv}
\begin{array}{l}
\frac{1}{2} \|\bar{V}_{n}(T)\|_{L^{2}(\Omega)}^{2} + \lambda \int_{0}^{T} \|\bar{V}_{n}(\tau)\|_{L^{2}(\Omega)}^{2}{\rm d} \tau + \int_{0}^{T} \int_{\Gamma_{\rm out}} Q(x,\tau) \bar{V}_{n}^{2}(x,\tau){\rm d} x {\rm d} \tau \\ \\
+ D_{\rm S} \int_{0}^{T} \|\nabla \bar{V}_{n}(\tau) \|_{L^{2}(\Omega)}^{2} {\rm d} \tau =\int_{0}^{T} \int_{\Omega} \mathbf{Q}(x,\tau) \bar{V}_{n}(x,\tau) \nabla \bar{V}_{n}(x,\tau) {\rm d}x {\rm d} \tau \\
\\
+ \int_{0}^{T} e^{-\lambda \tau}\int_{\Omega} \Big(\mu(Z_{\rm n}(x,\tau))B_{\rm Z_{\rm n}}(x,\tau)- \mu(Z(x,\tau))B_{\rm Z}(x,\tau) \Big) \bar{V}_{n}(x,\tau) {\rm d} x {\rm d}\tau .
\end{array}
\end{equation}

Similarly, if we multiply the second equation in (\ref{variacionalconlambda}) by $\bar{W}_{n}$, we have
\begin{equation}\label{variacionalw}
\begin{array}{l}
\frac{1}{2} \|\bar{W}_{n}(T)\|_{L^{2}(\Omega)}^{2} + \lambda \int_{0}^{T} \|\bar{W}_{n}(\tau)\|_{L^{2}(\Omega)}^{2}{\rm d} \tau + \int_{0}^{T} \int_{\Gamma_{\rm out}} Q(x,\tau) \bar{W}_{n}^{2}(x,\tau){\rm d} x {\rm d} \tau\\ \\
 + D_{\rm B} \int_{0}^{T} \|\nabla \bar{W}_{n}(\tau) \|_{L^{2}(\Omega)}^{2} {\rm d} \tau  = \int_{0}^{T} \int_{\Omega} \mathbf{Q}(x,\tau) \bar{W}_{n}(x,\tau) \nabla \bar{W}_{n}(x,\tau) {\rm d}x {\rm d} \tau\\
\\
+ \int_{0}^{T} e^{-\lambda \tau} \int_{\Omega} \Big(\mu(Z(x,\tau))B_{\rm Z}(x,\tau)- \mu(Z_{\rm n}(x,\tau))B_{\rm Z_{\rm n}}(x,\tau)\Big ) \bar{W}_{n}(x,\tau) {\rm d} x {\rm d}\tau .
\end{array}
\end{equation}

Summing equations (\ref{variacionalv}) and (\ref{variacionalw}) it follows:
\begin{equation}\label{variacionalsuma}
\begin{array}{l}
\frac{1}{2} \Big(\|\bar{V}_{n}(T)\|_{L^{2}(\Omega)}^{2}+\|\bar{W}_{n}(T)\|_{L^{2}(\Omega)}^{2}\Big) + \lambda \int_{0}^{T} (\|\bar{V}_{n}(\tau)\|_{L^{2}(\Omega)}^{2} + \|\bar{W}_{n}(\tau)\|_{L^{2}(\Omega)}^{2}){\rm d} \tau\\
\\
+ \int_{0}^{T} \|Q(\cdot,\tau)\|_{L^{\infty}(\bar{\Omega})} \Big(\|\bar{V}_{n}(\tau)\|_{L^{2}(\Gamma_{\rm out})}^{2}+ \|\bar{W}_{n}(\tau)\|_{L^{2}(\Gamma_{\rm out})}^{2} \Big) {\rm d} \tau \\
\\
+ \int_{0}^{T} (D_{\rm S} \|\nabla \bar{V}_{n}(\tau)\|_{L^{2}(\Omega)}^{2} + D_{\rm B} \|\nabla \bar{W}_{n}(\tau)\|_{L^{2}(\Omega)}^{2}){\rm d} \tau \\
\\
= \int_{0}^{T} \int_{\Omega} \mathbf{Q}(x,\tau) \Big(\bar{V}_{n}(x,\tau)\nabla \bar{V}_{n}(x,\tau) + \bar{W}_{n}(x,\tau) \nabla \bar{W}_{n}(x,\tau)\Big) {\rm d}x {\rm d} \tau \\
\\
+ \int_{0}^{T} e^{-\lambda \tau} \int_{\Omega} \Big(\mu(Z(x,\tau))B_{\rm Z}(x,\tau)- \mu(Z_{\rm n}(x,\tau))B_{\rm Z_{\rm n}}(x,\tau)\Big) \Big(\bar{W}_{n}(x,\tau)-\bar{V}_{n}(x,\tau)\Big) {\rm d} x {\rm d}\tau
\end{array}
\end{equation}

For the last term in (\ref{variacionalsuma}) we have that
\begin{equation*}
\begin{array}{l}
\int_{0}^{T} e^{-\lambda \tau} \int_{\Omega} \Big(\mu(Z(x,\tau))B_{\rm Z}(x,\tau)- \mu(Z_{\rm n}(x,\tau))B_{\rm Z_{\rm n}}(x,\tau)\Big)\Big(\bar{W}_{n}(x,\tau)-\bar{V}_{n}(x,\tau)\Big) {\rm d} x {\rm d} \tau\\
\\
= \int_{0}^{T} e^{-\lambda \tau} \int_{\Omega} \mu(Z(x,\tau))\Big(B_{\rm Z}(x,\tau)-B_{\rm Z_{\rm n}}(x,\tau)\Big) \Big(\bar{W}_{n}(x,\tau)-\bar{V}_{n}(x,\tau)\Big) {\rm d} x {\rm d}\tau\\
\\
 + \int_{0}^{T}e^{-\lambda \tau}\int_{\Omega} B_{\rm Z_{\rm n}}(x,\tau) \Big(\mu(Z(x,\tau))-\mu(Z_{\rm n}(x,\tau))\Big) \Big(\bar{W}_{n}(x,\tau)-\bar{V}_{n}(x,\tau)\Big) {\rm d} x {\rm d} \tau\\
\\
\leq  \frac{3}{2}\|\mu\|_{L^{\infty}(\mathds{R})} \int_{0}^{T} \|\bar{W}_{n}(\tau)\|_{L^{2}(\Omega)}^{2} {\rm d} \tau + \frac{1}{2}\|\mu\|_{L^{\infty}(\mathds{R})} \int_{0}^{T} \|\bar{V}_{n}(\tau)\|_{L^{2}(\Omega)}^{2} {\rm d} \tau\\
\\
+ \int_{0}^{T} |\mu(Z(x,\tau))-\mu(Z_{\rm n}(x,\tau)) | | B_{\rm Z_{\rm n}}(x,\tau)||\bar{W}_{n}(x,\tau)-\bar{V}_{n}(x,\tau)| {\rm d} x {\rm d} \tau. \\
\end{array}
\end{equation*}

Moreover, by applying Young's Inequality (see e.g., \cite{evans2010partial}) with $\epsilon_{1}>0$, which will be chosen below, it follows
\\
\\
$\int_{0}^{T}\int_{\Omega}  \mathbf{Q}(x,\tau) \bar{V}_{n}(x,\tau) \nabla \bar{V}_{n}(x,\tau){\rm d} x {\rm d} \tau$ 
\\
\\
$\hspace*{2cm}\leq \|Q\|_{L^{\infty}(\bar{\Omega}\times(0,T))} \int_{0}^{T} (\epsilon_{1} \|\bar{V}_{n}(\tau) \|_{L^{\rm 2}(\Omega)}^{2} + \frac{1}{4\epsilon_{1}} \| \nabla \bar{V}_{n} (\tau)\|_{L^{\rm 2}(\Omega)}^{2} ) {\rm d} \tau .$\\
\\
We apply the same reasoning for $\bar{W}_{n}$ with some positive constant $\epsilon_{2}>0$. 
\\
\\
Coming back to (\ref{variacionalsuma}) it follows that
\begin{equation}\label{variacionalsuma2}
\begin{array}{l}
\frac{1}{2} \Big(\|\bar{V}_{n}(T)\|_{L^{2}(\Omega)}^{2}+\|\bar{W}_{n}(T)\|_{L^{2}(\Omega)}^{2}\Big) + \int_{0}^{T} Q(x,\tau) (\|\bar{V}_{n}(\tau)\|_{L^{2}(\Gamma_{\rm out})}^{2}+ \|\bar{W}_{n}(\tau)\|_{L^{2}(\Gamma_{\rm out})}^{2}) {\rm d} \tau
\\
\\
+(D_{\rm S} - \frac{\|Q\|_{L^{\infty}(\bar{\Omega}\times(0,T))}}{4\epsilon_{1}}) \int_{0}^{T}\|\nabla \bar{V}_{n}(\tau)\|_{L^{2}(\Omega)}^{2}{\rm d} \tau + (D_{\rm B}- \frac{\|Q\|_{L^{\infty}(\bar{\Omega}\times(0,T))}}{4\epsilon_{2}})\int_{0}^{T} \|\nabla \bar{W}_{n}(\tau)\|_{L^{2}(\Omega)}^{2}){\rm d} \tau \\
\\
+ \underbrace{\Big(\lambda - \epsilon_{1}\|Q\|_{L^{\infty}(\bar{\Omega}\times(0,T))} - \frac{\|\mu\|_{L^{\infty}(\mathds{R})} }{2}}_{:=C}\Big ) \int_{0}^{T} \|\bar{V}_{n}(\tau)\|_{L^{2}(\Omega)}^{2}{\rm d} \tau\\
\\
+ (\lambda - \epsilon_{2}\|Q\|_{L^{\infty}(\bar{\Omega}\times(0,T))} - \frac{3}{2}\|\mu\|_{L^{\infty}(\mathds{R})})\int_{0}^{T} \|\bar{W}_{n}(\tau)\|_{L^{2}(\Omega)}^{2}{\rm d} \tau\\
\\
\leq \int_{0}^{T} \int_{\Omega} |\mu(Z(x,\tau))-\mu(Z_{n}(x,\tau)) | | B_{\rm Z_{\rm n}}(x,\tau)||(\bar{W}_{n}(x,\tau)-\bar{V}_{n}(x,\tau))| {\rm d} x {\rm d} \tau.
\end{array}
\end{equation}

If $\epsilon_{1}$, $\epsilon_{2}$ and $\lambda$ are chosen such that $\epsilon_{1} \geq \frac{\|Q\|_{L^{\infty}(\bar{\Omega}\times(0,T))}}{4D_{\rm S}}$, $\epsilon_{2} \geq \frac{\|Q\|_{L^{\infty}(\bar{\Omega}\times(0,T))}}{4D_{\rm B}}$ and 
$$\lambda > \|Q\|_{L^{\infty}(\bar{\Omega}\times(0,T))}\max(\epsilon_{1},\epsilon_{2}) + \frac{3}{2}\|\mu\|_{L^{\infty}(\mathds{R})},$$ 
one has
\begin{equation}\label{desigualdadfinal1}
\begin{array}{l}
\int_{0}^{T} \| \bar{V}_{n}(\tau) \|_{L^{2}(\Omega)}^{2}{\rm d} \tau \leq 2 \int_{0}^{T} \int_{\Omega} |\mu(Z(x,\tau))-\mu(Z_{\rm n}(x,\tau)) | | B_{\rm Z_{\rm n}}(x,\tau)||(\bar{W}_{n}(x,\tau)-\bar{V}_{n}(x,\tau))| {\rm d} x {\rm d} \tau.
\end{array}
\end{equation}

To prove that the right hand side of (\ref{desigualdadfinal1}) converges to $0$ as $n \rightarrow \infty$, we use the following steps:
\begin{enumerate}
\item Since $\|Z_{n}-Z\|_{L^{2}(\Omega \times (0,T))} \stackrel{n \rightarrow \infty}{\longrightarrow}0$, using Theorem \ref{fuertepuntual}, there exists a subsequence $\{ Z_{n_{k}} \}_{k} \subset \{Z_{n}\}_{n}$ such that $Z_{\rm n_{\rm k}} \rightarrow Z$ a.e. in $\Omega \times (0,T)$. Then, since $\mu$ is continuous, $\mu(Z_{\rm n_{k}}) \rightarrow \mu(Z)$ a.e. in $\Omega \times (0,T)$. For simplicity, we denote $\{Z_{n_{k}}\}_{k}= \{ Z_{k} \}_{k}$.

\item Since $\| \mu(Z_{k}) \|_{L^{\infty}(Q)} \leq \|\mu\|_{L^{\infty}(\mathds{R})}<+\infty$, by applying Theorem \ref{acotadaconvergente} using that $L^{1}(\Omega \times (0,T))$ is separable and $(L^{1}(\Omega \times (0,T)))'= L^{\infty}(\Omega \times (0,T))$ ), there exists a subsequence $\{ \mu(Z_{k_{j}}) \}_{j}$ weak-$\ast$ convergent to some $\omega \in L^{\infty}(\Omega \times (0,T))$. For simplicity, we denote $\{ Z_{k_{j}}\}_{j}=\{ Z_{j}\}_{j}$. 
\end{enumerate}

Due to steps 1 and 2, we conclude that $\{ \mu(Z_{\rm j}) \}_{j}$ is weak-$\ast$ convergent to $\mu(Z)$.

\begin{enumerate}\setcounter{enumi}{2}
\item $B_{\rm Z_{j}} \in K$, since $(S_{\rm Z_{j}}, B_{\rm Z_{j}})$ is solution of (\ref{generico}) with $c=\mu(Z_{j})$. Moreover, since $K \subset X$ is compact, there exists a subsequence $\{B_{Z_{j_{i}}} \}_{i} \subset \{ B_{\rm Z_{j}} \}_{j}$ such that there exist some $B \in X$ fulfilling $\| B_{\rm Z_{j_{i}}} - B \|_{X} \stackrel{i \rightarrow \infty}{\longrightarrow}0$. For simplicity, we denote $\{ Z_{j_{i}}\}_{i}=\{ Z_{i}\}_{i}$.

\item We define $$\bar{K}= \{ z \in W(0,T,H^{1}(\Omega),(H^{1}(\Omega))') \text{ }:\text{ } \|z\|_{W(0,T,H^{1}(\Omega),(H^{1}(\Omega))')} \leq 4\bar{C} \},$$ 
where $\bar{C}$ is the constant appearing in the definition of $K$ in \eqref{defK}. Notice that $\bar{K}$ is a compact set of $X$  (see Lemma \ref{aubin}). Since $W_{i}-V_{i}= B_{\rm Z_{i}}-S_{\rm Z_{i}} - B_{\rm Z}+S_{\rm Z} \in \bar{K}$, using the same reasoning as the one followed above, one obtains that there exists a subsequence $\{ W_{i_{r}}-V_{i_{r}} \}_{r} \subset \{W_{i} - V_{i}\}_{i}$ and $P \in X$ such that $\| (W_{i_{r}}-V_{i_{r}}) - P\|_{X} \stackrel{r \rightarrow \infty}{\longrightarrow}0$. For simplicity, we denote $\{Z_{i_{r}}\}_{r}= \{Z_{r}\}_{r}$. 
\end{enumerate}

By steps 3 and 4, we conclude that $B_{\rm r} (W_{r}-V_{r}) \subset L^{1}(\Omega \times (0,T))$ and $ \|B_{\rm r} (W_{r}-V_{r})-BP\|_{L^{1}(\Omega \times (0,T)) } \stackrel{r \rightarrow \infty}{\longrightarrow} 0$.\\

Furthermore, since $\{Z_{r}\}_{r} \subset \{Z_{j}\}_{j}$, it also follows that $\{ \mu(Z_{\rm r}) \}_{r}$ is weak-$\ast$ convergent to $\mu(Z)$. Using Theorem \ref{convergenciaproducto}, if follows that
\begin{equation}
\begin{array}{l}
\int_{0}^{T} \int_{\Omega} \underbrace{|\mu(Z(x,\tau))-\mu(Z_{\rm r}(x,\tau)) |}_{L^{\infty}(Q)} \underbrace{| B_{\rm Z_{r}}(x,\tau)||(\bar{W}_{r}(x,\tau)-\bar{V}_{r}(x,\tau))|}_{L^{1}(Q)} {\rm d} x {\rm d} \tau\\
\\
\hspace*{4cm}\stackrel{r \rightarrow \infty}{\longrightarrow} \int_{0}^{T} \int_{\Omega} 0 \cdot B(x,\tau) \cdot P(x,\tau) {\rm d}\tau {\rm d} x.
\end{array}
\end{equation}

From (\ref{desigualdadfinal1}), this implies that 
$$\int_{0}^{T}\int_{\Omega} e^{-2\lambda \tau} |S_{\rm Z_{r}}(x,\tau) - S_{\rm Z}(x,\tau) |^{2}{\rm d} x{\rm d} \tau \stackrel{r \rightarrow \infty}{\longrightarrow}0,$$
but since $\min_{\tau \in [0,T]} e^{-2\lambda \tau} = e^{-2\lambda T}$, one has that 
$$\|S_{\rm Z_{r}}-S_{Z}\|_{L^{2}(\Omega\times (0,T))}  \stackrel{r \rightarrow \infty}{\longrightarrow}0.$$

Finally, we prove that $\|S_{\rm Z_{n}}-S_{Z}\|_{L^{2}(\Omega\times (0,T))}  \stackrel{n \rightarrow \infty}{\longrightarrow}0$ (convergence of the whole sequence instead of subsequence) by reduction to absurdum. Let us assume that this is not true. Then, there exists $\epsilon>0$ and a subsequence $\{ S_{\rm Z_{n_{l}}} \}_{l} \subset \{ S_{\rm Z_{n}} \}_{n}$ such that
\begin{equation}\label{paracontradiccion}
\|S_{\rm Z_{n_{l}}} - S_{Z}\|_{L^{2}(\Omega\times (0,T))}>\epsilon, \hspace{1cm} \forall l \in \mathds{N}.
\end{equation}
If we now proceed as above, we can find a subsection $\{ S_{\rm Z_{n_{m}}} \}_{m}  \subset \{ S_{\rm Z_{n_{l}}} \}_{l} $ such that 
$$ \|S_{\rm Z_{n_{m}}}-S_{Z}\|_{L^{2}(\Omega\times (0,T))}  \stackrel{m \rightarrow \infty}{\longrightarrow}0,$$
which contradicts (\ref{paracontradiccion}).       
\end{proof}

Now, we are interested in studying the nonnegativity and boundedness properties of solutions $B$ and $S$.

\begin{thm}[Nonnegativity and boundedness of $B$]\label{bpositiva} 
Under assumptions of Theorem \ref{existencianolineal}:
\begin{enumerate}[(i)]
\item If $B_{\rm init} \geq 0$ in $\Omega$, then $B\geq 0$ in $\Omega\times(0,T)$.
\item If $B_{\rm init} \in L^{\infty}(\Omega)$, then $B(x,t) \leq \|B_{\rm init}\|_{L^{\infty}(\Omega)} e^{\|\mu\|_{L^{\infty}(\mathds{R})} t}$ a.e. $(x,t) \in \Omega\times(0,T)$.
\end{enumerate}
\end{thm}
\begin{proof}
We define the new variables $B^{+}=\max(B,0)$ and $B^{-}=-\min(B,0)$, then $B=B^{+} - B^{-}$ and the first statement of Theorem \ref{bpositiva} can be reformulated as
$$ B^{-}(x,0)=0 \text{ in }\Omega \Rightarrow B^{-}(x,t)=0 \hspace{1cm}\text{ in } \Omega \times (0,T).$$
Multiplying the second equation of (\ref{advdifreac}) by $B^{-}$ and integrating, one obtains
\begin{equation*}
\begin{array}{l l}
\frac{1}{2} \int_{0}^{t}  \frac{{\rm d}}{{\rm d} \tau} \| B^{-}(\tau) \|_{\rm L^{2}(\Omega)}^{2} {\rm d} \tau & = \int_{0}^{t} \int_{\Omega} \mathbf{Q}(x,\tau)B^{-}(x,\tau)\nabla B^{-}(x,\tau) {\rm d} x {\rm d} \tau \\
\\
&-\int_{0}^{t} \int_{\Gamma_{\rm out}} Q(x,\tau) (B^{-}(x,\tau))^{2} {\rm d}x {\rm d} \tau -  \int_{0}^{t} \int_{\Omega} D_{\rm B} (\nabla B^{-}(x,\tau))^{2}{\rm d} x {\rm d} \tau \\
\\
&+ \int_{0}^{t} \int_{\Omega} \mu(S(x,\tau))B^{-}(x,\tau)^{2} {\rm d}x {\rm d} \tau. \\
\end{array}
\end{equation*}
Applying Young's inequality with $\epsilon>0$ (that will be specified below), one has:
\begin{equation*}
\begin{array}{l l}
\frac{1}{2} \int_{0}^{t}  \frac{{\rm d}}{{\rm d} \tau} \| B^{-}(\tau) \|_{\rm L^{2}(\Omega)}^{2} {\rm d} \tau \leq &  (\epsilon \|Q\|_{L^{\infty}(\bar{\Omega}\times(0,T))} - D_{\rm B} ) \int_{0}^{t} \| \nabla B^{-}(\tau)\|_{L^{2}(\Omega)}^{2} {\rm d} \tau \\
\\
& +(\frac{\|Q\|_{L^{\infty}(\bar{\Omega}\times(0,T))}}{4\epsilon}+\|\mu\|_{L^{\infty}(\mathds{R})}) \int_{0}^{t}  \| B^{-}(\tau)\|_{L^{2}(\Omega)}^{2} {\rm d} \tau.\\
\end{array}
\end{equation*}
Choosing $\epsilon$ such that $\epsilon\|Q\|_{L^{\infty}(\bar{\Omega}\times(0,T))}  - D_{\rm B} \leq 0$ and applying Gronwall's Inequality in its integral form (see e.g., \cite{teschl}), one has:
$$ \| B^{-}(t) \|_{\rm L^{2}(\Omega)}^{2} \leq \underbrace{\| B^{-}(0) \|_{\rm L^{2}(\Omega)}^{2}}_{=0 \text{ by hypothesis}} \underbrace{e^{  2(\frac{\|Q\|_{L^{\infty}(\bar{\Omega}\times(0,T))}}{4\epsilon}+\|\mu\|_{L^{\infty}(\mathds{R})})t}}_{\geq 0}=0.$$

Consequently $B^{-}=0$ in $\Omega\times (0,T)$ and the statement (i) of the theorem is proved.\\

Now, we denote $U(x,t)= \|B_{\rm init} \|_{L^{\infty}(\Omega)} e^{\|\mu\|_{L^{\infty}(\mathds{R})} t} - B(x,t)$. We want to prove that $U(x,t)\geq 0$ in $\Omega\times(0,T)$. It fulfills \\
\begin{equation}\label{sistemau} 
\left \{
\begin{array}{l r}
U_{\rm t}= {\rm div}(D_{\rm B} \nabla U - \mathbf{Q}U) + \mu(S)U + \alpha e^{\|\mu\|_{L^{\infty}(\mathds{R})} t}  & \text{in } \Omega\times(0,T), \\
\\ 
U(x,0)=\|B_{\rm init}\|_{L^{\infty}(\Omega)}-B_{\rm init}(x) & \text{in } \Omega, \\ 
\\
\mathbf{n}\cdot (D_{\rm B}\nabla U - \mathbf{Q}U)=  Q(t)\|B_{\rm init} \|_{L^{\infty}(\Omega)} e^{\|\mu\|_{L^{\infty}(\mathds{R})} t} & \text{in } \Gamma_{\rm in} \times(0,T), \\ 
\\

\mathbf{n}\cdot (D_{\rm B}\nabla U - \mathbf{Q}U) =0 & \text{in }\Gamma_{\rm wall}\times(0,T),\\
\\ 
\mathbf{n}\cdot (D_{\rm B}\nabla U )= 0 & \text{in } \Gamma_{\rm out} \times(0,T), \\ 
\\ 
\end{array} \right.
\end{equation}
where $\alpha= ( \|\mu\|_{L^{\infty}(\mathds{R})}-\mu(S)) \|B_{\rm init}\|_{L^{\infty}(\Omega)}$. We define the new variables $U^{+}=\max(U,0)$ and $U^{-}=-\min(U,0)$, and proceeding as we did previously with $B$, it follows that
$$\| U^{-}(t) \|_{\rm L^{2}(\Omega)}^{2} \leq \| U^{-}(0) \|_{L^{2}(\Omega)}^{2} e^{2( \|\mu\|_{L^{\infty}(\Omega)} + \frac{\|Q\|_{L^{\infty}(\bar{\Omega}\times(0,T))}}{4\epsilon})t},$$
where $\epsilon$ is such that $\epsilon \|Q\|_{L^{\infty}(\bar{\Omega}\times(0,T))} - D_{\rm B} \leq 0$. Since $U(x,0)\geq 0$, then $\| U^{-}(0) \|_{\rm L^{2}(\Omega)}^{2}=0$ and, consequently, $U^{-}=0$ in $\Omega\times (0,T)$ and the statement (ii) of the theorem is proved.
\end{proof}

\begin{thm}[Nonnegativity and boundedness of $S$]\label{spositiva} 
Under assumptions of Theorem \ref{existencianolineal} and Theorem \ref{bpositiva}-(ii), if $S_{\rm e}\geq 0$ and $S_{\rm init} \geq 0$ in $\Omega$, $\mu$ is lipschitz and $\mu(0)=0$, then $S\geq 0$ in $\Omega\times(0,T)$. Furthermore,  if $S_{\rm init} \in L^{\infty}(\Omega)$, $S_{\rm e} \in L^{\infty}(0,T)$ and $\mu(z)> 0$ for $z> 0$, then $S \leq \max (\|S_{\rm init}\|_{\rm L^{\infty}(\Omega)}, \|S_{\rm e}\|_{\rm L^{\infty}(0,T)})$ in $\Omega\times(0,T)$.
\end{thm}
\begin{proof}
We define the new variables $S^{+}=\max(S,0)$ and $S^{-}=-\min(S,0)$. Then, multiplying the first equation of (\ref{advdifreac}) by $S^{-}$ and integrating it follows\\
\begin{equation}\label{positividads1}
\begin{array}{l}
\frac{1}{2} \int_{0}^{t} \frac{{\rm d}}{{\rm d} \tau} \| S^{-}(\tau) \|_{\rm L^{2}(\Omega)}^{2} {\rm d} \tau= \int_{0}^{t} \int_{\Omega} \mathbf{Q}(x,\tau)S^{-}(x,\tau)\nabla S^{-}(x,\tau) {\rm d} x {\rm d} \tau \\
\\
- \int_{0}^{t} \int_{\Omega} D_{\rm S} (\nabla S^{-}(x,\tau))^{2}{\rm d} x {\rm d} \tau  + \int_{0}^{t} \int_{\Omega} \mu(S(x,\tau))B(x,\tau)S^{-}(x,\tau) {\rm d}x {\rm d} \tau \\
\\
- \int_{0}^{t} \int_{\Gamma_{\rm in}} Q(x,\tau) S_{\rm e}(\tau) S^{-}(x,\tau) {\rm d}x {\rm d} \tau -\int_{0}^{t} \int_{\Gamma_{\rm out}} Q(x,\tau) (S^{-}(x,\tau))^{2} {\rm d}x {\rm d} \tau. 
\end{array}
\end{equation}
Under the hypothesis formulated on $\mu$, there exists a constant $C_{\rm L}$ such that\\
\\
$\begin{array}{l l} |\int_{0}^{t} \int_{\Omega} \mu(S(x,\tau))B(x,\tau)S^{-}(x,\tau) {\rm d}x {\rm d} \tau| & \leq C_{\rm L}\int_{0}^{t} \int_{\Omega} |S(x,\tau)| |B(x,\tau)| S^{-}(x,\tau) {\rm d}x {\rm d} \tau\\
\\
& \leq C_{\rm L} \|B\|_{L^{\infty}(\Omega\times(0,T))} \int_{0}^{t} \int_{\Omega} (S^{-}(x,\tau))^{2}{\rm d}x {\rm d} \tau.\end{array}$\\
\\
Furthermore, since $S_{\rm e}$, $Q$ and $S^{-}$ are nonnegative, from equation (\ref{positividads1}) one obtains
\begin{equation}\label{positividads2}
\begin{array}{l}
\frac{1}{2} \int_{0}^{t} \frac{{\rm d}}{{\rm d} \tau} \| S^{-}(\tau) \|_{\rm L^{2}(\Omega)}^{2} {\rm d} \tau \leq C_{\rm L} \|B\|_{L^{\infty}(\Omega\times(0,T))} \int_{0}^{t} \int_{\Omega} (S^{-}(x,\tau))^{2} {\rm d}x {\rm d} \tau \\
\\
-  \int_{0}^{t} \int_{\Omega} D_{\rm S} (\nabla S^{-}(x,\tau))^{2}{\rm d} x {\rm d} \tau  + \int_{0}^{t} \int_{\Omega} \mathbf{Q}(x,\tau)S^{-}(x,\tau)\nabla S^{-}(x,\tau) {\rm d} x {\rm d} \tau.
\end{array}
\end{equation}
Moreover, applying Young's inequality (see e.g., \cite{evans2010partial}) with $\epsilon>0$ (that will be specified below), one has:
\begin{equation*}
\begin{array}{l}
\frac{1}{2} \int_{0}^{t} \frac{{\rm d}}{{\rm d} \tau} \| S^{-}(\tau) \|_{\rm L^{2}(\Omega)}^{2}{\rm d}\tau  \leq  (\epsilon \|Q\|_{L^{\infty}(\bar{\Omega}\times(0,T))} - D_{\rm S} ) \int_{0}^{T} \| \nabla S^{-}(\tau)\|_{L^{2}(\Omega)}^{2} {\rm d} \tau
\\
\\
\hspace*{2cm} +(\frac{\|Q\|_{L^{\infty}(\bar{\Omega}\times(0,T))}}{4\epsilon}+ C_{\rm L} \|B\|_{L^{\infty}(\Omega\times(0,T))}) \int_{0}^{T}  \| S^{-}(\tau)\|_{L^{2}(\Omega)}^{2} {\rm d} \tau. \\
\\
\end{array}
\end{equation*}
Choosing $\epsilon$ such that $\epsilon \|Q\|_{L^{\infty}(\bar{\Omega}\times(0,T))} - D_{\rm S} \leq 0$ and applying Gronwall's Inequality in its integral form (see e.g., \cite{teschl}), one has:\\
\\
$ \| S^{-}(t) \|_{\rm L^{2}(\Omega)}^{2} \leq \underbrace{\| S^{-}(0) \|_{\rm L^{2}(\Omega)}^{2}}_{=0 \text{ by hypothesis }} e^{2(\frac{\|Q\|_{L^{\infty}(\bar{\Omega}\times(0,T))}}{4\epsilon}+ C_{\rm L} \|B\|_{L^{\infty}(\Omega\times(0,T))})t }=0$.\\
Consequently $S^{-}=0$ in $\Omega\times (0,T)$ and the first statement of the theorem is proved.
\\
\\
Now, we denote $\beta =\max (\|S_{\rm init} \|_{L^{\infty}(\Omega)},\|S_{\rm e} \|_{L^{\infty}(0,T)} )$ and $U(x,t)= \beta - S(x,t)$. We want to prove that $U(x,t)\geq 0$ in $\Omega\times(0,T)$. It fulfills
\begin{equation}\label{sistemau2} 
\left \{
\begin{array}{l r}
U_{\rm t}= {\rm div}(D_{\rm S} \nabla U - \mathbf{Q}U) + \mu(S)B   & \text{in } \Omega\times(0,T), \\
\\ 
U(x,0)=\beta-S_{\rm init}(x) & \text{in } \Omega, \\ 
\\
\mathbf{n}\cdot (D_{\rm S}\nabla U - \mathbf{Q}U)=  Q ( \beta - S_{\rm e}) & \text{in } \Gamma_{\rm in} \times(0,T), \\ 
\\
\mathbf{n}\cdot (D_{\rm S}\nabla U - \mathbf{Q}U) =0 & \text{in }\Gamma_{\rm wall}\times(0,T),\\
\\ 
\mathbf{n}\cdot (D_{\rm S}\nabla U )= 0 & \text{in } \Gamma_{\rm out} \times(0,T). \\ 
\\ 
\end{array} \right.
\end{equation}
We define the new variables $U^{+}=\max(U,0)$ and $U^{-}=-\min(U,0)$, and using the same reasoning as the one followed in Theorem \ref{bpositiva} one has 
$$ \| U^{-}(t) \|_{\rm L^{2}(\Omega)}^{2} \leq \| U^{-}(0) \|_{L^{2}(\Omega)}^{2} e^{\frac{\|Q\|_{L^{\infty}(\bar{\Omega}\times(0,T))}}{2\epsilon}t},$$
where $\epsilon$ such that $\epsilon \|Q\|_{L^{\infty}(\bar{\Omega}\times(0,T))} - D_{\rm S} \leq 0$. \\
Since $U(x,0)\geq 0$, then $\| U^{-}(0) \|_{\rm L^{2}(\Omega)}^{2}=0$ and, consequently, $U^{-}=0$ in $\Omega\times (0,T)$ and the second statement of the theorem is proved.
\end{proof}

\begin{rmk}
: Notice that we assume that $Q$, $S_{\rm e}$, $B_{\rm init}$ and $S_{\rm init}$ are nonnegative and essentially bounded because of their physical meaning. The assumption $\mu(0)=0$ is due to the fact that if there is no substrate concentration, no reaction is produced; the assumption $\mu(z)> 0$ if $z> 0$ follows from the fact that if there is substrate, the reaction makes the substrate concentration decrease and the biomass concentration increase (see System (\ref{advdifreac})). These two assumptions are commonly used in bioreactor theory (see e.g.,~\cite{Rapaport20081052}). Furthermore, the assumption of considering that function $\mu$ is essentially bounded is a caused by the fact that microorganisms have a maximum specific growth rate.   
\end{rmk}

Finally, we prove the uniqueness of solution of System (\ref{advdifreac}).
\begin{thm}[Uniqueness of solution]\label{unicidadnolineal2}
Under the hypothesis of Theorem \ref{bpositiva} and if $\mu$ is Lipschitz, then System (\ref{advdifreac}) has a unique weak solution $(S,B)$.
\end{thm}

\begin{proof}
Let us assume that $(S_{1},B_{1})$ and $(S_{2},B_{2})$ are two different weak solutions of System (\ref{advdifreac}). We denote $V=S_{1}-S_{2}$, $W=B_{1}-B_{2}$ and $\bar{V}=e^{-\lambda t}V$, $\bar{W}=e^{-\lambda t}W$, where $\lambda>0$ will be chosen later. Proceeding as in previous theorems, we can obtain the following energy estimate:
\begin{equation}\label{intento1s}
\begin{array}{l}
\frac{1}{2} \|\bar{V}(T)\|_{L^{2}(\Omega)}^{2} + \lambda \int_{0}^{T} \|\bar{V}(\tau)\|_{L^{2}(\Omega)}^{2}{\rm d} \tau + \int_{0}^{T} \int_{\Gamma_{\rm out}} Q(x,\tau) \bar{V}(x,\tau)^{2}{\rm d} x {\rm d} \tau \\ \\
+ D_{\rm S}\int_{0}^{T} \|\nabla \bar{V}(\tau) \|_{L^{2}(\Omega)}^{2} {\rm d} \tau   = \int_{0}^{T} \int_{\Omega} \mathbf{Q} \bar{V}(x,\tau) \nabla \bar{V}(x,\tau) {\rm d}x {\rm d} \tau
\\
\\
+ \underbrace{\int_{0}^{T} e^{-\lambda \tau}\int_{\Omega} \Big(\mu(S_{\rm 2}(x,\tau))B_{\rm 2}(x,\tau) - \mu(S_{\rm 1}(x,\tau))B_{\rm 1}(x,\tau) \Big) \bar{V}(x,\tau) {\rm d} x {\rm d}\tau}_{(I)}.
\end{array}
\end{equation}
Now,
\begin{equation*}
\begin{array}{l l}
(I) = & \int_{0}^{T}e^{-\lambda \tau} \int_{\Omega} \mu(S_{\rm 1}(x,\tau))\Big(B_{\rm 2}(x,\tau)-B_{\rm 1}(x,\tau)\Big) \bar{V}(x,\tau){\rm d}x {\rm d} \tau\\
\\
 &+ \int_{0}^{T}e^{-\lambda \tau}\int_{\Omega} \Big(\mu(S_{\rm 2}(x,\tau)) - \mu(S_{\rm 1}(x,\tau))\Big )B_{\rm 2}(x,\tau) \bar{V}(x,\tau){\rm d} x {\rm d} \tau.\\
\end{array}
\end{equation*}
Moreover, since $B_{\rm 2} \in L^{\infty}(\Omega \times (0,T))$ (see Theorem \ref{bpositiva}), and using the fact that $\mu$ is Lipschitz, there exists a constant $C_{\rm L}>0$ such that  
\begin{equation*}
\begin{array}{l l}
(I) & \leq \|\mu\|_{L^{\infty}(\mathds{R})} \int_{0}^{T}\int_{\Omega}|\bar{V}(x,\tau)| |\bar{W}(x,\tau)|{\rm d} x{\rm d} \tau\\
\\
 & +  C_{\rm L} \int_{0}^{T} e^{-\lambda t} \int_{\Omega} |S_{\rm 2}(x,\tau) - S_{\rm 1}(x,\tau)| |B_{\rm 2}(x,\tau)| |\bar{V}(x,\tau)| {\rm d} x{\rm d} \tau\\
\\
& \leq \frac{\|\mu\|_{L^{\infty}(\mathds{R})}}{2} \int_{0}^{T}(\|\bar{V}(\tau)\|_{L^{2}(\Omega)}^{2}+ \|\bar{W}(\tau)\|_{L^{2}(\Omega)}^{2}){\rm d} \tau + C_{\rm L} \|B_{\rm 2}\|_{L^{\infty}(\Omega \times (0,T))} \int_{0}^{T} \|\bar{V}(\tau)\|_{L^{2}(\Omega)}^{2} {\rm d} \tau \\
\\
& = \frac{\|\mu\|_{L^{\infty}(\mathds{R})}}{2} \int_{0}^{T}\|\bar{W}(\tau)\|_{L^{2}(\Omega)}^{2}{\rm d} \tau + ( \frac{\|\mu\|_{L^{\infty}(\mathds{R})}}{2} + C_{\rm L} \|B_{\rm 2}\|_{L^{\infty}(\Omega \times (0,T))}) \int_{0}^{T} \|\bar{V}(\tau)\|_{L^{2}(\Omega)}^{2} {\rm d} \tau.\\
\end{array}
\end{equation*}
Coming back to (\ref{intento1s}) and applying Holder's and Young's inequality (see e.g., \cite{evans2010partial}) with $\epsilon_{1}>0$ (that will be chosen later), one has:
\begin{equation}\label{intento2s}
\begin{array}{l}
\frac{1}{2} \|\bar{V}(T)\|_{L^{2}(\Omega)}^{2} + \Big(\lambda - \epsilon_{1} \|Q\|_{L^{\infty}(\bar{\Omega}\times(0,T))} - \frac{\|\mu\|_{L^{\infty}(\mathds{R})}}{2} - C_{\rm L} \|B_{\rm 2}\|_{L^{\infty}(\Omega \times (0,T))} \Big)\int_{0}^{T} \|\bar{V}(\tau)\|_{L^{2}(\Omega)}^{2}{\rm d} \tau  \\
\\
+ \int_{0}^{T} \int_{\Gamma_{\rm out}} Q(x,\tau) \bar{V}(x,\tau)^{2}{\rm d} x {\rm d} \tau + \Big(D_{\rm S} - \frac{\|Q\|_{L^{\infty}(\bar{\Omega}\times(0,T))}}{4\epsilon_{1}} \Big)\int_{0}^{T}\|\nabla \bar{V}(\tau)\|_{L^{2}(\Omega)}^{2}{\rm d} \tau \\
\\
\leq \frac{\|\mu\|_{L^{\infty}(\mathds{R})}}{2} \int_{0}^{T} \|\bar{W}(\tau)\|_{L^{2}(\Omega)}^{2}{\rm d} \tau. 
\end{array}
\end{equation}
Proceeding analogously, we obtain the following energy estimate
\begin{equation}\label{intento1b}
\begin{array}{l}
\frac{1}{2} \|\bar{W}(T)\|_{L^{2}(\Omega)}^{2} + \lambda \int_{0}^{T} \|\bar{W}(\tau)\|_{L^{2}(\Omega)}^{2}{\rm d} \tau + \int_{0}^{T} \int_{\Gamma_{\rm out}} Q(x,\tau) \bar{W}(x,\tau)^{2}{\rm d} x {\rm d} \tau \\ \\
+ D_{\rm B} \int_{0}^{T} \|\nabla \bar{W}(\tau) \|_{L^{2}(\Omega)}^{2} {\rm d} \tau  =\int_{0}^{T} \int_{\Omega} \mathbf{Q} \bar{W}(x,\tau) \nabla \bar{W}(x,\tau) {\rm d}x {\rm d} \tau \\
\\
+ \underbrace{\int_{0}^{T} e^{-\lambda \tau}\int_{\Omega} (\mu(S_{\rm 1}(x,\tau))B_{\rm 1}(x,\tau)-\mu(S_{\rm 2}(x,\tau))B_{\rm 2}(x,\tau)) \bar{W}(x,\tau) {\rm d} x {\rm d}\tau}_{(II)}.
\end{array}
\end{equation}

Now,
\begin{equation*}
(II) =  \int_{0}^{T}\int_{\Omega} \mu(S_{\rm 1}(x,\tau))\bar{W}(x,\tau)^{2}{\rm d}x {\rm d} \tau + \int_{0}^{T}e^{-\lambda \tau}\int_{\Omega} \Big(\mu(S_{\rm 1}(x,\tau)) - \mu(S_{\rm 2}(x,\tau))\Big)B_{\rm 2}(x,\tau) \bar{W}(x,\tau){\rm d} x {\rm d} \tau.
\end{equation*}
Since $\mu$ is Lipschitz and $B_{\rm 2} \in L^{\infty}(\Omega \times (0,T))$ (see Theorem \ref{bpositiva}) one has 
\begin{equation*}
\begin{array}{l}
(II) \leq \|\mu\|_{L^{\infty}(\mathds{R})}\int_{0}^{T}\|\bar{W}(\tau)\|_{L^{2}(\Omega)}^{2}{\rm d} \tau + C_{L} \|B_{\rm 2}\|_{L^{\infty}(\Omega \times (0,T))} \int_{0}^{T}\int_{\Omega} |\bar{W}(x,\tau)\bar{V}(x,\tau)|{\rm d} x {\rm d} \tau.
\end{array}
\end{equation*}
Applying Young's inequality (see e.g., \cite{evans2010partial}) with $\epsilon=\frac{1}{2}$, one obtains 
\begin{equation*}
\begin{array}{l}
(II) \leq \|\mu\|_{L^{\infty}(\mathds{R})}\int_{0}^{T}\|\bar{W}(\tau)\|_{L^{2}(\Omega)}^{2}{\rm d} \tau + C_{\rm L}\|B_{\rm 2}(\tau)\|_{L^{\infty}(\Omega \times (0,T))}  \int_{0}^{T} (\frac{\|\bar{W}(\tau)\|_{L^{2}(\Omega)}^{2}}{2} + \frac{\|\bar{V}(\tau)\|_{L^{2}(\Omega)}^{2}}{2}) {\rm d} \tau.
\end{array}
\end{equation*}

Coming back to equation (\ref{intento1b}), it follows that
\begin{equation}\label{intento2b}
\begin{array}{l}
\frac{1}{2} \|\bar{W}(T)\|_{L^{2}(\Omega)}^{2} + \int_{0}^{T} \int_{\Gamma_{\rm out}} Q(x,\tau) \bar{W}(x,\tau)^{2}{\rm d} x {\rm d} \tau + \Big(D_{\rm B} - \frac{\|Q\|_{L^{\infty}(\bar{\Omega}\times(0,T))}}{4\epsilon_{2}} \Big)\int_{0}^{T}\|\nabla \bar{W}(\tau)\|_{L^{2}(\Omega)}^{2}{\rm d} \tau \\
\\
+\Big(\lambda - \epsilon_{2} \|Q\|_{L^{\infty}(\bar{\Omega}\times(0,T))} - \|\mu\|_{L^{\infty}(\mathds{R})} - \frac{C_{L} \|B_{2}\|_{L^{\infty}(\Omega\times(0,T))}}{2}\Big)\int_{0}^{T} \|\bar{W}(\tau)\|_{L^{2}(\Omega)}^{2}{\rm d} \tau\\
\\
\leq \frac{C_{L} \|B_{\rm 2}\|_{L^{\infty}(\Omega\times(0,T))}}{2} \int_{0}^{T} \|\bar{V}(\tau)\|_{L^{2}(\Omega)}^{2}{\rm d} \tau. 
\end{array}
\end{equation}

Finally, adding equations (\ref{intento2s}) and (\ref{intento2b}), we obtain
\begin{equation}
\begin{array}{l}
\frac{1}{2} (\|\bar{V}(t)\|_{L^{2}(\Omega)}^{2} + \|\bar{W}(t)\|_{L^{2}(\Omega)}^{2})\\
\\ + (\lambda - \epsilon_{1} \|Q\|_{L^{\infty}(\bar{\Omega}\times(0,T))} - \frac{\|\mu\|_{L^{\infty}(\mathds{R})}}{2} - \frac{3 C_{L} \|B_{\rm 2}\|_{L^{\infty}(\Omega\times(0,T))}}{2}) \int_{0}^{T} \|\bar{V}(\tau)\|_{L^{2}(\Omega)}^{2}{\rm d} \tau \\
\\
+ (\lambda - \epsilon_{2} \|Q\|_{L^{\infty}(\bar{\Omega}\times(0,T))} - \frac{3}{2}\|\mu\|_{L^{\infty}(\mathds{R})} - \frac{C_{L} \|B_{\rm 2}\|_{L^{\infty}(\Omega\times(0,T))}}{2})\int_{0}^{T} \|\bar{W}(\tau)\|_{L^{2}(\Omega)}^{2}{\rm d} \tau \\
\\
+(D_{\rm B} - \frac{\|Q\|_{L^{\infty}(\bar{\Omega}\times(0,T))}}{4\epsilon_{2}})\int_{0}^{T}\|\nabla \bar{W}(\tau)\|_{L^{2}(\Omega)}^{2}{\rm d} \tau  + (D_{\rm S} - \frac{\|Q\|_{L^{\infty}(\bar{\Omega}\times(0,T))}}{4\epsilon_{1}}) \int_{0}^{T}\|\nabla \bar{V}(\tau)\|_{L^{2}(\Omega)}^{2}{\rm d} \tau \leq 0. 
\end{array}
\end{equation}

Choosing $\epsilon_{1} > \frac{\|Q\|_{L^{\infty}(\bar{\Omega}\times(0,T))}}{4D_{\rm S}}$, $\epsilon_{2} > \frac{\|Q\|_{L^{\infty}(\bar{\Omega}\times(0,T))}}{4D_{\rm B}}$ and 
$$\lambda> \frac{3C_{\rm L}\|B_{2}\|_{L^{\infty}(\Omega\times (0,T))}}{2} + \max\{\epsilon_{1},\epsilon_{2}\} \|Q\|_{L^{\infty}(\bar{\Omega}\times(0,T))} + \frac{3}{2}\|\mu\|_{L^{\infty}(\mathds{R})},$$ it follows that $\|\bar{W}\|_{L^{2}(0,T,H^{1}(\Omega))}+\|\bar{V}\|_{L^{2}(0,T,H^{1}(\Omega))}=0$, which implies that $\bar{W}=\bar{V}=0$ in $\Omega \times (0,T)$. Consequently $S_{\rm 1}=S_{\rm 2}$ and $B_{\rm 1}=B_{\rm 2}$ in $\Omega \times (0,T)$ and we have proved the statement of Theorem \ref{unicidadnolineal2}.
\end{proof}
\section{Conclusion}
In this work, we have focused on the modeling of a continuous flow bioreactor in which a biomass and a substrate are interacting. We have carried out a mathematical analysis of the system of partial differential equations appearing in the model. We have stated the definition of solution and we have proved theoretical results showing the existence and uniqueness of solution under the assumptions of both linear and nonlinear reaction terms. We have also shown non-negativity and boundedness results for the solution. The results shown in this work are of interest for the study of this type of bioreactor models, their design and the optimization of the corresponding processes (see, e.g., \cite{bellorivas,AIC}).
\section*{Acknowledgements}
This work was carried out thanks to the financial support of  the Spanish ``Ministry of Economy and Competitiveness'' under project MTM2011-22658; the research group MOMAT (Ref. 910480) supported by ``Banco Santander'' and ``Universidad Complutense de Madrid''; and the ``Junta de Andaluc\'ia'' and the European Regional Development Fund through project P12-TIC301.
\section*{Appendix}
We consider $V$, $H$ Hilbert spaces such that $V \subset H$ and $V$ is dense on $H$. If we identify $V'$ with the dual of $V$, it follows that
$$V \subset H \subset V'.$$
Furthermore, if $\mathcal{V}=L^{2}(0,T,V)$ and $\mathcal{H}=L^{2}(0,T,H)$, one has that $\mathcal{V}'=L^{2}(0,T,V')$. 
\\
\\
We also consider the space
$$W(0,T,V,V')= \{ u | u \in L^{2}(0,T,V), \frac{{\rm d} u}{{\rm d} t} \in L^{2}(0,T,V')\},$$
with the norm
$$\|u\|_{W(0,T,V,V')} = ( \int_{0}^{T} \|u(t)\|_{V}^{2} {\rm d} t + \int_{0}^{T} \| \frac{{\rm d} u}{{\rm d} t} (t) \|_{V'}^{2} {\rm d} t)^{\frac{1}{2}}.$$

\begin{thm}[Theorem 3.1 and Proposition 2.1 in \cite{lions}]\label{acotadoporcontinuas}
$$W(0,T,V,V') \subset C^{0}([0,T],H).$$
\end{thm}

Let $a(t,\cdot,\cdot)$ a bilinear and continuous form on $V$, a.e.$t\in (0,T)$, satisfying the following conditions: 
\begin{equation}\label{hipotesisauv1}
\left \{\begin{array}{l}
\forall \text{ }u,v \in V \hspace{1cm}\text{the function } t\rightarrow a(t,u,v) \text{ is measurable and }
\\
\\
\exists \text{ }c\in \mathds{R}:\text{ } |a(t,u,v)|\leq c\|u\| \|v\| \hspace{1cm} \forall u,v\in V, \text{ a.e}.t \in [0,T].
\end{array} \right.
\end{equation}

There exists $\lambda, \alpha>0$ such that 
\begin{equation}\label{hipotesisauv2}
a(t,v,v) + \lambda|v|^{2} \geq \alpha \|v\|^{2} \hspace{1cm} \forall v\in V,\text{ }a.e.t\in [0,T].
\end{equation}

Since a.e. $t\in (0,T)$, the form $v \rightarrow a(t,u,v)$ is continuous on $V$, there exists $A(t)u \in V'$ such that
$$a(t,u,v)=<A(t)u,v>_{V' \times V},$$
which defines
$$A(t) \in \mathcal{L}(V,V').$$

Let us consider the following evolution problem
\begin{equation}\label{problemalions}
\left \{ \begin{array}{l}
\text{Find } u \in W(0,T,V,V') \text{ such that } \\
\\
A(t)u + \frac{{\rm d} u}{{\rm d} t} =f, \hspace{1cm} \text{ where } f \in L^{2}(0,T,V'),\\
\\
y(0)=u_{0}, \hspace{1cm} \text{ where } u_{0} \in H.
\end{array} \right. \end{equation}
\begin{thm}[Theorem 1.2, Chapter III~\cite{lionscontrol}]\label{teoremalions}
Assume Hypothesis (\ref{hipotesisauv1}) and (\ref{hipotesisauv2}) hold. Then Problem (\ref{problemalions}) has a unique weak solution.
\end{thm}
\begin{lem}[Aubin-Lions Compactness Lemma]\label{aubin}
Let $X \subset B \subset Y$ Banach spaces such that the inclusion $X \subset B$ is a compact embedding. Then, for any $1 < p < \infty$, $1 \leq q \leq \infty$, the space
$$\{ f : f \in L^{p}(0,T,X) \text{ and } \frac{{\rm d} f}{{\rm d} t} \in L^{q}(0,T,Y) \}$$
is compact embedded in $L^{p}(0,T,B)$.\\
\\
Particularly, if $q=p=2$, $X=H^{1}(\Omega)$ and $B=L^{2}(\Omega)$ and $Y=(H^{1}(\Omega))'$ it follows that
$$W(0,T,H^{1}(\Omega),(H^{1}(\Omega))') \subset L^{2}(0,T,L^{2}(\Omega),$$
with compact embedding. 
\end{lem}
\begin{thm}[Theorem 3.18, \cite{brezis2010}]\label{acotadaconvergente}
Let $X$ a separable space and $\{ f_{\rm n} \}_{n} \subset X' $ a bounded sequence, then there exists a subsequence $\{ f_{\rm n_{\rm k}} \}_{k}$ that converges in the weak-$\ast$ topology to some $f \in X'$. 
\end{thm}
\begin{thm}[Theorem 4.9, \cite{brezis2010}]\label{fuertepuntual}
Let $1\leq p \leq \infty$. If $\{ f_{\rm n} \}_{n} \subset L^{p}(Q)$ and $f \in L^{p}(Q)$, such that $\|f_{\rm n} - f\|_{L^{p}(Q)} \rightarrow 0$, then there exists a subsequence $\{ f_{\rm n_{\rm k}} \}_{k}$ such that $f_{\rm n_{\rm k}} \rightarrow f$ almost everywhere in $Q$.
\end{thm}
\begin{thm}[Proposition 3.13(iv), \cite{brezis2010}]\label{convergenciaproducto}
If $\{ f_{\rm n} \}_{n} \subset X'$ converges to $f \in X'$ in the weak-$\ast$ topology, $\{ x_{\rm n} \}_{n} \subset X$, $x\in X$  such that $\|x_{\rm n} - x \|_{X}\rightarrow 0$, then
$$<f_{\rm n},x_{\rm n}>_{X'\times X} \rightarrow <f,x>_{X'\times X}.$$
\end{thm}

\bibliographystyle{spmpsci}      


%
%

\end{document}